\documentclass{amsart}
\usepackage{amssymb}
\usepackage{mathrsfs}
\usepackage{stmaryrd}
\usepackage[all]{xy}

\usepackage{graphicx}
\usepackage{caption}
\usepackage{subcaption}

\newcommand{\C}{\mathbb{C}}
\newcommand{\Ql}{\mathbb{Q}_\ell}
\newcommand{\Z}{\mathbb{Z}}
\newcommand{\Fpb}{\overline{\F}_p}

\newcommand{\K}{\mathbb{K}}
\newcommand{\F}{\mathbb{F}}
\renewcommand{\O}{\mathbb{O}}
\newcommand{\E}{\mathbb{E}}
\newcommand{\bk}{\Bbbk}
\newcommand{\Gm}{\mathbb{G}_{\mathrm{m}}}
\newcommand{\Ga}{\mathbb{G}_{\mathrm{a}}}

\newcommand{\fg}{\mathfrak{g}}
\newcommand{\ft}{\mathfrak{t}}
\newcommand{\fh}{\mathfrak{h}}
\newcommand{\cB}{\mathscr{B}}
\newcommand{\scP}{\mathscr{P}}

\newcommand{\Gv}{{\check G}}
\newcommand{\Bv}{{\check B}}
\newcommand{\Pv}{{\check P}}
\newcommand{\Tv}{{\check T}}
\newcommand{\cBv}{{\check \cB}}
\newcommand{\scPv}{{\check \scP}}

\newcommand{\Db}{D^{\mathrm{b}}}
\newcommand{\Kb}{K^{\mathrm{b}}}
\newcommand{\Dbc}{D^{\mathrm{b}}_{\mathrm{c}}}
\newcommand{\Perv}{\mathsf{P}}
\newcommand{\Proj}{\mathsf{Proj}}
\newcommand{\Tilt}{\mathsf{Tilt}}

\newcommand{\Mod}{\mathsf{mod}}
\newcommand{\gmod}{\mathsf{gmod}}
\newcommand{\lh}{\text{-}}
\newcommand{\Parity}{\mathsf{Parity}}
\newcommand{\cS}{\mathcal{S}}

\newcommand{\For}{\mathsf{For}}
\newcommand{\Av}{\mathrm{Av}}
\newcommand{\av}{\mathrm{av}}
\newcommand{\AS}{\mathrm{AS}}
\newcommand{\D}{\mathbb{D}}

\newcommand{\cF}{\mathcal{F}}
\newcommand{\cG}{\mathcal{G}}
\newcommand{\cH}{\mathcal{H}}
\newcommand{\cL}{\mathcal{L}}
\newcommand{\cM}{\mathcal{M}}
\newcommand{\cN}{\mathcal{N}}
\newcommand{\cP}{\mathcal{P}}
\newcommand{\cT}{\mathcal{T}}
\newcommand{\cX}{\mathcal{X}}

\newcommand{\IC}{\mathcal{IC}}

\newcommand{\pH}{{}^p\!\mathcal{H}}

\newcommand{\ICv}{{\check \IC}}
\newcommand{\dv}{{\check \Delta}}
\newcommand{\nv}{{\check \nabla}}
\newcommand{\cPv}{{\check \cP}}
\newcommand{\cTv}{{\check \cT}}
\newcommand{\cEv}{\check{\mathcal{E}}}

\newcommand{\simto}{\xrightarrow{\sim}}
\DeclareMathOperator{\End}{End}
\DeclareMathOperator{\Hom}{Hom}
\DeclareMathOperator{\Ext}{Ext}
\newcommand{\id}{\mathrm{id}}
\DeclareMathOperator{\codim}{codim}

\makeatletter
\def\lotimes{\@ifnextchar_{\@lotimessub}{\@lotimesnosub}}
\def\@lotimessub_#1{\mathchoice{\mathbin{\mathop{\otimes}^L}_{#1}}%
  {\otimes^L_{#1}}{\otimes^L_{#1}}{\otimes^L_{#1}}}
\def\@lotimesnosub{\mathbin{\mathop{\otimes}^L}}
\makeatother

\newcommand{\GL}{\mathrm{GL}}

\newcommand{\bfY}{\mathbf{Y}}

\newcommand{\scS}{\mathscr{S}}
\newcommand{\scT}{\mathscr{T}}
\newcommand{\hS}{\widehat{S}}

\newcommand{\us}{\underline{s}}
\newcommand{\ut}{\underline{t}}
\newcommand{\la}{\langle}
\newcommand{\ra}{\rangle}

\newcommand{\iv}{{\check \imath}}
\newcommand{\gr}{\mathrm{gr}}

\newtheorem*{thm*}{Theorem}
\numberwithin{equation}{section}
\newtheorem{thm}{Theorem}[section]
\newtheorem{lem}[thm]{Lemma}
\newtheorem{prop}[thm]{Proposition}
\newtheorem{cor}[thm]{Corollary}
\theoremstyle{definition}
\newtheorem{defn}[thm]{Definition}
\theoremstyle{remark}
\newtheorem{rmk}[thm]{Remark}

\title[Modular perverse sheaves on flag varieties I]{Modular perverse sheaves on flag varieties I:\\ tilting and parity sheaves}

\author{Pramod N. Achar}
\address{Department of Mathematics\\
  Louisiana State University\\
  Baton Rouge, LA 70803\\
  U.S.A.}
\email{pramod@math.lsu.edu}

\author{Simon Riche}
\address{Universit{\'e} Blaise Pascal - Clermont-Ferrand II, Laboratoire de Math{\'e}matiques, CNRS, UMR 6620, Campus universitaire des C{\'e}zeaux, F-63177 Aubi{\`e}re Cedex, France
}
\email{simon.riche@math.univ-bpclermont.fr}

\makeatletter
\let\@wraptoccontribs\wraptoccontribs
\makeatother
\contrib[with a joint appendix with]{Geordie Williamson}

\subjclass[2010]{14M15, 14F05, 20G40.}
\thanks{P.A. was supported by NSF Grant No.~DMS-1001594.  S.R. was supported by ANR Grants No.~ANR-09-JCJC-0102-01, ANR-2010-BLAN-110-02 and ANR-13-BS01-0001-01.}

\begin{document}

\begin{abstract}
In this paper we prove that the category of parity complexes on the flag variety of a complex connected reductive group $G$ is a ``graded version'' of the category of tilting perverse sheaves on the flag variety of the dual group $\Gv$, for any field of coefficients whose characteristic is good for $G$. We derive some consequences on Soergel's modular category $\mathcal{O}$, and on multiplicities and decomposition numbers in the category of perverse sheaves.
\end{abstract}

\maketitle

\section{Introduction}
\label{sec:intro}

\subsection{}

This paper is the first in a series devoted to investigating the structure of the category of Bruhat-constructible perverse sheaves on the flag variety of a complex connected reductive algebraic group, with coefficients in a field of positive characteristic. In this part, adapting some constructions of Bezrukavnikov--Yun~\cite{by} in the characteristic $0$ setting, we show that in good characteristic, the category of parity sheaves on the flag variety of a reductive group is a ``graded version'' of the category of tilting perverse sheaves on the flag variety of the Langlands dual group. We also derive a number of interesting consequences of this result, in particular on the computation of multiplicities of simple perverse sheaves in standard perverse sheaves, on Soergel's ``modular category $\mathcal{O}$,'' and on decomposition numbers.

\subsection{Some notation}
\label{ss:intro-notation}

Let $G$ be a complex connected reductive algebraic group, and let $T \subset B \subset G$ be a maximal torus and a Borel subgroup. The choice of $B$ determines a choice of positive roots of $(G,T)$, namely those appearing in $\mathrm{Lie}(B)$.
Consider also the Langlands dual data $\Tv \subset \Bv \subset \Gv$. That is, $\Gv$ is a complex connected reductive group, and we are given an isomorphism $X^*(T) \cong X_*(\Tv)$ which identifies the roots of $G$ with the coroots of $\Gv$ (and the positive roots determined by $B$ with the positive coroots determined by $\Bv$).

We are interested in the varieties $\cB:=G/B$ and $\cBv:=\Gv/\Bv$, in the derived categories
\[
\Db_{(B)}(\cB,\bk), \qquad \text{resp.} \qquad \Db_{(\Bv)}(\cBv,\bk)
\]
of sheaves of $\bk$-vector spaces on these varieties, constructible with respect to the stratification by $B$-orbits, resp.~$\Bv$-orbits (where $\bk$ is field), and in their abelian subcategories
\[
\Perv_{(B)}(\cB,\bk), \qquad \text{resp.} \qquad \Perv_{(\Bv)}(\cBv,\bk)
\]
of perverse sheaves (for the middle perversity). The category $\Perv_{(B)}(\cB,\bk)$ is highest weight, with simple objects $\{\IC_w, \, w \in W\}$, standard objects $\{\Delta_w, \, w \in W\}$, costandard objects $\{\nabla_w, \, w \in W\}$, indecomposable projective objects $\{\cP_w, \, w \in W\}$ and indecomposable tilting objects $\{\cT_w, \, w \in W\}$ naturally parametrized by the Weyl group $W$ of $(G,T)$. Similar remarks apply of course to $\Perv_{(\Bv)}(\cBv,\bk)$, and we denote the corresponding objects by $\ICv_w$, $\dv_w$, $\nv_w$, $\cPv_w$, $\cTv_w$. (Note that the Weyl group of $(\Gv,\Tv)$ is canonically identified with $W$.)

\subsection{The case $\bk=\C$}
\label{ss:intro-properties}

These categories have been extensively studied in the case $\bk=\C$: see in particular \cite{bgs, bbm, by}. To state some of their properties we need some notation. We will denote by $\mathsf{IC}_{(\Bv)}(\cBv,\C)$ the additive category of semisimple objects in $\Db_{(\Bv)}(\cBv,\C)$ (i.e.~the full subcategory whose objects are direct sums of shifts of simple perverse sheaves). If $\mathsf{A}$ is an abelian category, we will denote by $\Proj\lh\mathsf{A}$ the additive category of projective objects in $\mathsf{A}$. And finally, if $\mathsf{A}$, $\mathsf{B}$ are additive categories, if $T$ is an autoequivalence of $\mathsf{A}$, and if $\For \colon \mathsf{A} \to \mathsf{B}$ is a functor endowed with an isomorphism $\varepsilon \colon \For \circ T \simto \For$, we will say that $\For$ \emph{realizes $\mathsf{A}$ as a graded version of $\mathsf{B}$} if $\For$ is essentially surjective and, for any $M,N$ in $\mathsf{A}$, the natural morphism
\begin{equation}
\label{eqn:graded-version}
\bigoplus_{n \in \Z} \Hom \bigl( M,T^n(N) \bigr) \to \Hom(\For M,\For N)
\end{equation}
induced by $\For$ and $\varepsilon$ is an isomorphism.

With these notations, some of the main properties of our categories can be roughly stated as follows.
\begin{enumerate}
\item
\label{it:property-BB}
(``Be{\u\i}linson--Bernstein localization'') 
There exists an equivalence of abelian categories $\Perv_{(B)}(\cB,\C) \cong \mathcal{O}_0(G)$, where $\mathcal{O}_0(G)$ is the principal block of the category $\mathcal{O}$ of the Lie algebra of $G$.
\item
\label{it:property-Soergel}
(``Soergel theory'')
There exists a functor $\nu \colon \mathsf{IC}_{(\Bv)}(\cBv,\C) \to \Proj\lh\mathcal{O}_0(G)$ which realizes $\mathsf{IC}_{(\Bv)}(\cBv,\C)$ (endowed with the shift autoequivalence $[1]$) as a graded version of $\Proj\lh\mathcal{O}_0(G)$.
\item
\label{it:property-KL}
(``Kazhdan--Lusztig conjecture'') 
The multiplicities $[\nabla_w : \IC_v]$ are determined by the specialization at $q=1$ of a canonical basis of the Hecke algebra $\mathcal{H}_W$ of $W$.
\item
\label{it:property-Koszul}
(``Koszul duality'')
There exists a triangulated category $D^{\mathrm{mix}}$ endowed with an autoequivalence and a diagram
\[
\Db_{(B)}(\cB,\C) \longleftarrow D^{\mathrm{mix}} \longrightarrow \Db_{(\Bv)}(\cBv,\C)
\]
where both functors are such that \eqref{eqn:graded-version} is an isomorphism for all $M, N$ (for a suitable $T$),
and where simple perverse sheaves on the left correspond to tilting perverse sheaves on the right. As a consequence, the category $\Perv_{(B)}(\cB,\C)$ is equivalent to the category of (ungraded) modules over a Koszul ring.
\item
\label{it:property-self-dual}
(``Koszul self-duality'')
The diagram in \eqref{it:property-Koszul} is symmetric in the sense that tilting perverse sheaves on the left also correspond to simple perverse shea\-ves on the right.
\item
\label{it:property-Ringel}
(``Ringel duality'')
There exists an autoequivalence of the triangulated category $\Db_{(B)}(\cB,\C)$ sending $\nabla_w$ to $\Delta_{w w_0}$ and $\cT_w$ to $\cP_{w w_0}$. (Here, $w_0$ is the longest element in $W$.) As a consequence, we have
\[
(\cT_w : \nabla_v) = (\cP_{w w_0} : \Delta_{w w_0}).
\]
\item
\label{it:property-formality}
(``formality'')
If we set $\IC_\cB:=\bigoplus_{w \in W} \IC_w$ and if we consider the graded algebra $E=\bigl( \bigoplus_{n \in \Z} \Ext^n_{\Db_{(B)}(\cB,\C)}(\IC_\cB,\IC_\cB[n]) \bigr)^{\mathrm{op}}$ as a differential graded algebra with trivial differential, then there exists an equivalence of triangulated categories
\[
\Db_{(B)}(\cB,\C) \cong E\lh\mathsf{dgDerf}
\]
where the right-hand side is the derived category of finitely generated differential graded $E$-modules.
\end{enumerate}

The goal of this series of papers is to give analogues of these properties in the case where $\bk$ is of characteristic $\ell>0$.

\subsection{Known results}
\label{ss:intro-known}

First, let us recall what is known about the properties of~\S\ref{ss:intro-properties} when $\C$ is replaced by a finite field $\bk$ of characteristic $\ell>0$. (This case will be referred to as the ``modular case,'' as opposed to the ``ordinary case'' when $\ell=0$.)

Property \eqref{it:property-Ringel} can be immediately generalized, with the same proof (see~\S\ref{ss:reminder-radon}). Property \eqref{it:property-Soergel} was generalized by Soergel in~\cite{soergel}. Here the main difference with the ordinary case appears: in the modular case the category $\mathsf{IC}_{(\Bv)}(\cBv,\bk)$ is not well behaved, and the ``nice'' additive category which should replace $\mathsf{IC}_{(\Bv)}(\cBv,\C)$ is the category $\Parity_{(\Bv)}(\cBv,\bk)$ of \emph{parity complexes} in the sense of~\cite{jmw}. With this replacement, $(2)$ still holds (when $\ell$ is bigger than the Coxeter number of $G$) when $\mathcal{O}_0(G)$ is replaced by Soergel's ``modular category $\mathcal{O}$,'' a certain subquotient of the category of rational representations of the reductive algebraic group over $\bk$ which has the same root datum as $G$.

Property \eqref{it:property-formality} was also generalized to the modular case (again, where simple perverse sheaves are replaced by parity sheaves) in~\cite{rsw}, under the assumption that $\ell$ is at least the number of roots of $G$ plus $2$. Using this result, a representation-theoretic analogue of \eqref{it:property-Koszul} (which can be obtained, in the ordinary case, by combining properties~\eqref{it:property-BB} and~\eqref{it:property-Koszul}) was also obtained in~\cite{rsw}, under the same assumptions.

In~\cite{rsw} a second, more technical, difference between the modular setting and the ordinary one appears, related to eigenvalues of the Frobenius. In fact, to obtain a ``formality'' statement as in \eqref{it:property-formality} one needs to introduce a differential graded algebra whose cohomology is $E$ and which is endowed with an additional $\Z$-grading. This additional grading is obtained using eigenvalues of a Frobenius action, which are all of the form $\overline{p}^n$ where $n \in \Z$ and $\overline{p}$ is the image of a fixed prime number $p \neq \ell$ in $\bk$. If $\ell=0$ then $\overline{p}^n \neq \overline{p}^m$ if $n \neq m$,
and one obtains directly the grading. But if $\ell>0$ this property no longer holds.
In~\cite{rsw} this difficulty was overcome, but at the price of unnecessary assumptions on $\ell$. 

\subsection{Main result}
\label{ss:intro-main-result}

In this paper we explain how to adapt properties \eqref{it:property-BB} and \eqref{it:property-KL} of~\S\ref{ss:intro-properties} to the modular setting (when $\ell$ is good for $G$). In the ordinary case, historically these questions were treated first, and Koszul duality was discovered later as a convenient way to express many nice features of this situation. In the modular case we will first establish a weak form of Koszul duality (which is a first step in the direction of properties \eqref{it:property-Koszul}, \eqref{it:property-self-dual} and \eqref{it:property-formality}, as  explained in \cite{ar}) and then deduce \eqref{it:property-BB} and \eqref{it:property-KL}.

In fact we construct a functor
\[
\nu \colon \Parity_{(\Bv)}(\cBv,\bk) \to \Tilt_{(B)}(\cB,\bk)
\]
from the additive category of Bruhat-constructible parity complexes on $\cBv$ to the additive category of tilting objects in~$\Perv_{(B)}(\cB,\bk)$, which realizes the former category as a graded version of the latter category, and which sends the indecomposable parity sheaf parametrized by $w$ to the tilting perverse sheaf parametrized by $w$. (See Theorem \ref{thm:main} for a more precise statement.)

Our construction is adapted from the main constructions in \cite{by}: we describe both categories in terms of some categories of ``Soergel modules'' for the coinvariant algebra, using a ``functor $\mathbb{H}$'' for parity sheaves and a ``functor $\mathbb{V}$'' for tilting perverse sheaves. The case of parity sheaves is a relatively easy generalization of \cite{by}. (A similar equivalence was already considered in \cite{soergel} in case $\ell$ is bigger than the Coxeter number of $G$.) In this case, the functor of tensoring with a ``basic'' Soergel bimodule associated with a simple reflection corresponds to a ``push-pull'' functor associated with the projection to the associated partial flag variety.

The case of tilting perverse sheaves is more subtle, and requires some new ideas, in particular to define what it means to ``take the logarithm of the monodromy'' (see~\S\ref{ss:monodromy}). In this case, the functor of tensoring with a ``basic'' Soergel bimodule corresponds to taking an ``averaging functor'' with respect to a ``Whittaker type'' action of a unipotent group. (It would have been possible to follow the proofs in~\cite{by} more closely, using an ``equivariant'' setting for parity sheaves, and a ``free monodromic'' setting for tilting perverse sheaves. However, to obtain simpler proofs we combine some constructions from~\cite{by} with some arguments from~\cite{soergel}.)

\subsection{Applications}
\label{ss:intro-applications}

As an application of our construction, in~\S\ref{ss:modular-O} we prove that, if $\ell$ is bigger than the Coxeter number of $G$, the abelian category $\Perv_{(B)}(\cB,\bk)$ is equivalent, as a highest weight category, to Soergel's modular category $\mathcal{O}$, thereby obtaining a modular analogue of property \eqref{it:property-BB}.

Regarding property \eqref{it:property-KL}, recall that the graded dimension of the cohomology of the stalk of a simple perverse sheaf $\IC_w$ at a $T$-fixed point corresponding to $v$ is given by the Kazhdan--Lusztig polynomial attached to $(v,w)$ up to some normalization (see~\cite{kl, springer}). Hence the multiplicity $[\nabla_w^\C : \IC_v^\C]$ is given by the value at $1$ of an \emph{inverse} Kazhdan--Lusztig polynomial. But the inversion formula for Kazhdan--Lusztig polynomials (which can be seen as a combinatorial manifestation of Koszul duality) implies that inverse Kazhdan--Lusztig polynomials are also Kazhdan--Lusztig polynomials (for the dual group); in other words, with standard notation (see~\S\S\ref{ss:background-tilt}--\ref{ss:background-parity} for details) we have
\[
[\nabla_w^\C : \IC_v^\C] = \dim \mathbb{H}^\bullet(\cBv_{w_0 w^{-1}}, \iv_{w_0 w^{-1}}^* \ICv{}_{w_0 v^{-1}}^\C).
\]
In~\S\ref{ss:multiplicities} we show that this formula still holds in the modular setting (in good characteristic), if $\ICv_{w_0 v^{-1}}$ is replaced by the corresponding parity sheaf $\cEv_{w_0 v^{-1}}$.
By definition (see~\cite{williamson2}) the basis of $\mathcal{H}_W$ determined by the graded dimensions of stalks of parity sheaves is the \emph{$\ell$-canonical basis}. Our result shows that this basis also describes composition multiplicities for the \emph{dual} group. Note that this $\ell$-canonical basis can be computed algorithmically,\footnote{This algorithm is due to G.~Williamson (unpublished). It is based on the interpretation of the local intersections forms appearing in the computation of convolutions of parity sheaves (see~\cite[\S 3.3]{jmw} in terms of the ``Soergel calculus'' of~\cite{elias-williamson}, and allows us to compute the $\ell$-canonical basis elements by induction on the Bruhat order.} so the same holds for the multiplicities $[\nabla_w : \IC_v]$.

This result can be deduced from the modular analogue of \eqref{it:property-Koszul} that will be proved in~\cite{ar}. But we need not wait for that: a direct proof, using only properties of the functor $\nu$, is also possible.  Indeed, by the usual reciprocity formula in the highest weight category $\Perv_{(B)}(\cB,\bk)$, we have $[\nabla_w : \IC_v] = (\cP_v : \Delta_w)$. Now by Ringel duality (see property \eqref{it:property-Ringel} in~\S\ref{ss:intro-properties}) this multiplicity can be computed in terms of multiplicities for tilting perverse sheaves. Finally, our functor $\nu$ allows us to express this multiplicity in terms of the cohomology of stalks of parity sheaves.

As a last application of our constructions, in~\S\ref{ss:decomposition} we prove that the decomposition matrices for parity sheaves, tilting perverse sheaves, projective perverse sheaves and intersection cohomology complexes are related in a very simple way.

\subsection{Perspectives}

In~\cite{ar, ar2} we use the functor $\nu$ of~\S\ref{ss:intro-main-result} to provide modular analogues of properties (4), (5) and (7) of \S\ref{ss:intro-properties}. Note that these constructions do not involve any consideration about Frobenius weights as mentioned in~\S\ref{ss:intro-known}. 
In fact the ``grading'' that shows up in \eqref{it:property-Koszul} and \eqref{it:property-self-dual} (and in a hidden way in \eqref{it:property-formality}) will come purely from the grading that shows up in \eqref{it:property-Soergel}.
We are even able to define a ``modular version'' of the category of mixed perverse sheaves considered in \cite[\S 4.4]{bgs} without using any theory of {\'e}tale sheaves over finite fields.

Regarding the last sentence in property (4),
it was expected for some time that if $\ell$ is not too small (say, bigger than the Coxeter number of $G$) then (Bruhat-constructible) parity sheaves on $\cB$ or $\cBv$ are just intersection cohomology complexes. (In fact, thanks to \cite{soergel}, this assertion is equivalent to the validity of Lusztig's conjecture on characters of modular simple representations of reductive algebraic groups ``around the Steinberg weight.'') However, recent results of Williamson \cite{williamson} have shown that this was too optimistic, even in the case $G=\GL_n(\C)$. As a consequence, one cannot hope that the category $\Perv_{(B)}(\cB,\bk)$ is governed by a Koszul ring unless $\ell$ is very large. (In this case, Koszulity was shown in \cite[\S 5.7]{rsw}; see also \cite{weidner}.) We will consider this question in \cite{ar,ar2}.

\subsection{Further notation and conventions}

The notations related to $G$ and $\Gv$ introduced in~\S\ref{ss:intro-notation} will be retained throughout the paper.  In addition, we will denote by $w_0$ the longest element in $W$, and by $U$ (resp.~$U^-$) the unipotent radical of $B$ (resp.~the opposite Borel subgroup).

Let us fix, once and for all, a prime number $\ell$ that is good for $G$ (or equivalently for $\Gv$).  Recall that this condition excludes $\ell=2$ if the root system $R$ of our group has a component not of type $\mathbf{A}$, $\ell=3$ if $R$ has a component of type $\mathbf{E}$, $\mathbf{F}$ or $\mathbf{G}$, and $\ell=5$ if $R$ has a component of type $\mathbf{E}_8$.
Fix a finite extension $\O$ of $\Z_\ell$.  Denote by $\K$ its field of fractions and by $\F$ its residue field. (Thus, $\ell$ is the characteristic of $\F$.)  Typically, we will use the letter $\E$ to denote any of $\K$, $\O$, or $\F$.  On various occasions we will have to use separate arguments for the cases of $\K$, $\O$, $\F$. In this case we will add superscripts or subscripts to the notations to emphasize the ring of coefficients.  In Sections~\ref{sec:tilting-V}--\ref{sec:proof}, we will assume in addition that $\#\F > 2$.

These three rings will serve as coefficients for derived categories of sheaves. Only field coefficients appeared in~\S\ref{ss:intro-notation}, but the same notions make sense for $\E = \O$ as well. (In this case, the perversity we use is the middle perversity $p$, and \emph{not} the perversity $p^+$ from~\cite[\S 3.3.4]{bbd}.) Additional background and references concerning the various classes of perverse sheaves, especially for the case $\E = \O$, will be given in Section~\ref{sec:main-result}.

We define the shift $\langle 1 \rangle$ in the category of $\Z$-graded $\E$-modules as follows: if $M=\oplus_{n \in \Z} M_n$ is a graded $\E$-module, then $M \langle 1 \rangle$ is the graded $\E$-module with $(M \langle 1 \rangle)_n=M_{n+1}$. Note that this convention is opposite to the convention chosen in \cite{rsw}. 
If $A$ is a noetherian algebra we will denote by $A\lh\Mod$ the category of finitely generated left $A$-modules. If $A$ is a \emph{$\Z$-graded} noetherian algebra we will denote by $A\lh\gmod$ the category of finitely generated $\Z$-graded left $A$-modules.

Throughout the paper, we use the convention that $\For$ generically represents an obvious forgetful functor (of the $\Z$-grading).

\subsection{Contents}

In Section~\ref{sec:main-result} we state our main result, and deduce the applications mentioned in~\S\ref{ss:intro-applications}.  The proof of the main result occupies Sections~\ref{sec:soergel-modules}--\ref{sec:proof}.  Section~\ref{sec:soergel-modules} introduces ``Soergel modules,'' which will provide the bridge between tilting perverse sheaves on $\cB$ and parity sheaves on $\cBv$. In Section~\ref{sec:hh} we relate parity sheaves to Soergel modules via the functor $\mathbb{H}$. In Section~\ref{sec:tilting-V} we relate tilting perverse sheaves to Soergel modules via the functor $\mathbb{V}$. We finish the proof of our main theorem in Section~\ref{sec:proof}.

The paper ends with two appendices. Appendix~\ref{sec:equivariant-Db} collects some well-known results on equivariant derived categories that we were not able to find in the literature, and Appendix~\ref{sec:tilting} (joint with Geordie Williamson) explains how to generalize standard results on tilting perverse sheaves to the case of integral coefficients.

\subsection{Acknowledgements}

This project was initially started as joint work with Geordie Williamson, and some of these results were even conjectured by him. Although he was not able to participate to the final stages of this work, his ideas, comments and enthusiasm were crucial for the development of our ideas. We also thank Patrick Polo for his interest and for checking some of our conjectures in interesting special cases. Finally we thank several anonymous referees for their helpful suggestions.

\section{Main result and applications}
\label{sec:main-result}

\subsection{Background on tilting sheaves}
\label{ss:background-tilt}

For any $w \in W$ one can consider the orbit
\[
\cB_w:=BwB/B
\]
and the inclusion $i_w \colon \cB_w \hookrightarrow \cB$.  Then the \emph{standard} and \emph{costandard} perverse sheaves on $\cB$ are defined as
\[
\Delta_w := i_{w!} \underline{\E}_{\cB_w}[\dim \cB_w], \qquad \nabla_w:= i_{w*} \underline{\E}_{\cB_w} [\dim \cB_w],
\]
where $\underline{\E}_{\cB_w}$ denotes the constant sheaf on $\cB_w$ with value $\E$.  (The fact that $\Delta_w$ and $\nabla_w$ are perverse sheaves follows from~\cite[Corollaire~4.1.3]{bbd} in the case $\E=\K$ or $\F$; in the case $\E=\O$ this property can be deduced from the case $\E=\F$, see~\cite[Lemma 2.1.2]{rsw}.\footnote{Here is an alternative argument, suggested by a referee. Recall that an object $\cF$ is in ${}^p \Db_{(B)}(\cB,\O)^{\leq 0}$ iff $\F(\cF)$ is in ${}^p \Db_{(B)}(\cB,\F)^{\leq 0}$. Using this observation and the results of~\cite{bbd} we deduce that $\Delta_w^\O$ and $\nabla_w^\O$ are in ${}^p \Db_{(B)}(\cB,\O)^{\leq 0}$ Using Verdier duality it follows that both of these objects are also in ${}^{p^+} \Db_{(B)}(\cB,\O)^{\geq 0}$. Hence (see~\cite[\S 3.3.4, Equation (ii)]{bbd}) they are torsion-free perverse sheaves.})

Recall that an object $\cT$ in $\Perv_{(B)}(\cB,\E)$ is called \emph{tilting} if it admits both a standard filtration (i.e.~a filtration with subquotients isomorphic to $\Delta_v$ for $v \in W$) and a costandard filtration (i.e.~a filtration with subquotients isomorphic to $\nabla_v$ for $v \in W$). We denote by $\Tilt_{(B)}(\cB,\E)$ the full additive subcategory of $\Perv_{(B)}(\cB,\E)$ consisting of tilting objects. It is well known in case $\E=\K$ or $\F$ (see~\cite{ringel, bbm}), and proved in Appendix \ref{sec:tilting} in case $\E=\O$, that this category is Krull--Schmidt, and that its indecomposable objects are parametrized by $W$; we denote by $\cT_w$ the indecomposable object associated with $w \in W$. We also denote by
\[
(\cT : \nabla_v)
\]
the number of times $\nabla_v$ appears in a costandard filtration of the tilting object $\cT$. (This number does not depend on the choice of the filtration; it is equal to the rank of the free $\E$-module $\Hom(\Delta_v,\cT)$.) The extension of scalars functors 
\begin{align*}
\F(-):=\F \lotimes_\O (-) \colon \Db_{(B)}(\cB,\O) &\to \Db_{(B)}(\cB,\F), \\
\K(-):=\K \otimes_\O (-) \colon \Db_{(B)}(\cB,\O) &\to \Db_{(B)}(\cB,\K)
\end{align*}
restrict to functors between the categories of tilting perverse sheaves, which we denote similarly.

\subsection{Background on parity sheaves}
\label{ss:background-parity}

On the dual side, for any $w \in W$ one can consider the orbit
\[
\cBv_w:=\Bv w \Bv / \Bv,
\]
and the inclusion $\iv_w \colon \cBv_w \hookrightarrow \cBv$. Recall (following \cite{jmw}) that an object ${\check \cF}$ of $\Db_{(\Bv)}(\cBv,\E)$ is called \emph{$*$-even} if for any $w \in W$ and $n \in \Z$ the sheaf $\cH^n(\iv_w^* {\check \cF})$ is zero if $n$ is odd, and a free $\E$-local system otherwise. Similarly, ${\check \cF}$ is called \emph{$!$-even} if
for any $w \in W$ and $n \in \Z$ the sheaf $\cH^n(\iv_w^! {\check \cF})$ is zero if $n$ is odd, and a free $\E$-local system otherwise. The object ${\check \cF}$ is called \emph{even} if it is both $*$-even and $!$-even, and \emph{odd} if $\cF[1]$ is even. Finally, an object is called a \emph{parity complex} if it is a direct sum of an even object and an odd object. We denote by $\Parity_{(\Bv)}(\cBv,\E)$ the full additive subcategory of $\Db_{(\Bv)}(\cBv,\E)$ consisting of parity complexes. This subcategory is stable under the shift $[1]$. It follows from \cite[\S 4.1]{jmw} that this category is Krull--Schmidt, and that its indecomposable objects are parametrized by $W \times \Z$. More precisely, for any $w \in W$ there exists a unique indecomposable object $\cEv_w$ in $\Parity_{(\Bv)}(\cBv,\E)$ which is supported on the closure of $\cBv_w$, and whose restriction to $\cBv_w$ is $\underline{\E}_{\cBv_w}[\dim \cBv_w]$. Such an object is called a \emph{parity sheaf}. Then any indecomposable object in $\Parity_{(\Bv)}(\cBv,\E)$ is a shift of a parity sheaf. Note finally that the functors $\F(-)$ and $\K(-)$ also restrict to functors between categories of parity complexes.

\subsection{Background on projective sheaves and Radon transform}
\label{ss:reminder-radon}

(This subsection is not needed for the statement of the main result, but it will be needed for the subsequent applications.)
It follows from \cite{bgs} in case $\E=\K$ or $\F$ and from \cite[\S 2.4]{rsw} in case $\E=\O$ that the category $\Perv_{(B)}(\cB,\E)$ has enough projective objects, that the full subcategory $\Proj_{(B)}(\cB,\E)$ of $\Perv_{(B)}(\cB,\E)$ consisting of projective objects is Krull--Schmidt, and that its indecomposable objects are naturally parametrized by $W$.  We denote by $\cP_w$ the projective object associated with $w \in W$. (If $\E=\K$ or $\F$ then by definition $\cP_w^\E$ is the projective cover of the simple object $\IC_w^\E$. If $\E=\O$ these objects are defined in \cite[\S 2.4]{rsw}; one can show in this case as well that $\cP_w^\O$ is the projective cover of $\IC_w^\O$.) It is known that each $\cP_w$ admits a standard filtration, i.e.~a filtration with subquotients isomorphic to $\Delta_v$ for $v \in W$. We denote by
\[
(\cP_w : \Delta_v)
\]
the number of times $\Delta_v$ appears in such a filtration. (This number does not depend on the filtration; it is equal to the rank of the free $\E$-module $\Hom(\cP_w,\nabla_v)$.)

Let $\mathscr{U}$ denote the open $G$-orbit in $\cB \times \cB$, and consider the diagram
\[
\xymatrix@R=0.2cm{
& \mathscr{U} \ar[ld]_-{p} \ar[rd]^-{q} & \\
\cB & & \cB,
}
\]
where $p$ and $q$ denote the projection on the first and second factors, respectively.
Following \cite{bb, bbm, yun} we consider the ``Radon transform'' (or ``geometric Ringel duality'')
\[
\mathsf{R} := q_! p^* [\dim \cB] \colon \Db_{(B)}(\cB,\E) \to \Db_{(B)}(\cB,\E).
\]
It is well known that this functor is an equivalence of triangulated categories, with inverse isomorphic to $p_* q^! [-\dim \cB]$; moreover, this equivalence satisfies
\[
\mathsf{R}(\nabla_w) \cong \Delta_{w w_0}, \qquad \mathsf{R}(\cT_w) \cong \cP_{w w_0}
\]
for all $w \in W$. (More precisely, these properties are proved in \cite{bbm, yun} in the case $\E=\K$. The same arguments apply to the case $\E=\F$. The case $\E=\O$ is not much more difficult; details are treated in~\S\ref{ss:radon}.)
In particular, this equivalence restricts to an equivalence of categories
\[
\mathsf{R} \colon \Tilt_{(B)}(\cB,\E) \simto \Proj_{(B)}(\cB,\E).
\]
We also deduce that for any $v,w \in W$ we have
\begin{equation}
\label{eqn:multiplicities-T-P}
(\cT_w : \nabla_v) = (\cP_{w w_0} : \Delta_{v w_0}).
\end{equation}

\subsection{Statement}
\label{ss:statement}

The following theorem is the main result of the paper.  

\begin{thm}
\label{thm:main}
Assume that $\ell$ is good for $G$, and that $\#\F > 2$. 

For $\E=\K$, $\O$ or $\F$ there exists a functor
\[
\nu_\E \colon \Parity_{(\Bv)}(\cBv,\E) \to \Tilt_{(B)}(\cB,\E)
\]
and an isomorphism of functors $\varepsilon \colon \nu_\E \circ [1] \simto \nu_\E$ such that the following hold.
\begin{enumerate}
\item 
For any $\cEv,{\check \cF}$ in $\Parity_{(\Bv)}(\cBv,\E)$, the functor $\nu_\E$ and the isomorphism $\varepsilon$ induce an isomorphism
\[
\bigoplus_{n \in \Z} \Hom \bigl( \cEv,{\check \cF}[n] \bigr) \simto \Hom \bigl( \nu_\E(\cEv),\nu_\E({\check \cF}) \bigr).
\]
\item
For any $w \in W$ we have $\nu_\E(\cEv_w) \cong \cT_{w^{-1}}$.
\item
For any $\cEv$ in $\Parity_{(\Bv)}(\cBv,\E)$ and $v \in W$ we have
\[
(\nu_\E(\cEv) : \nabla_v) = \mathrm{rk}_\E \bigl( \mathbb{H}^\bullet(\cBv_{v^{-1}}, \iv_{v^{-1}}^* \cEv) \bigr).
\]
\item
The functors $\nu_\E$ are compatible with extension of scalars in the sense that the following diagram commutes up to isomorphisms of functors:
\[
\xymatrix@C=1.5cm{
\Parity_{(\Bv)}(\cBv,\K) \ar[d]_-{\nu_\K}^-{\wr} & \Parity_{(\Bv)}(\cBv,\O) \ar[r]^-{\F(-)} \ar[l]_-{\K(-)} \ar[d]_-{\nu_\O}^-{\wr} & \Parity_{(\Bv)}(\cBv,\F) \ar[d]^-{\nu_\F}_-{\wr} \\
\Tilt_{(B)}(\cB,\K) & \Tilt_{(B)}(\cB,\O) \ar[r]^-{\F(-)} \ar[l]_-{\K(-)} & \Tilt_{(B)}(\cB,\F).
}
\]
\end{enumerate}
\end{thm}

(The requirement that $\#\F > 2$ arises because technical aspects of our study of $\Tilt_{(B)}(\cB,\E)$ require $\F$ to contain nontrivial roots of unity.  The proof of this theorem will be given in Sections~\ref{sec:soergel-modules}--\ref{sec:proof}.)

Properties $(1)$ and $(2)$ say that $\nu$ realizes $\Parity_{(\Bv)}(\cBv,\E)$ as a graded version of $\Tilt_{(B)}(\cB,\E)$, in the sense of~\S\ref{ss:intro-properties}. Then property $(3)$ says that the cohomologies of stalks of parity sheaves provide \emph{graded versions} of the multiplicities of costandard objects in tilting objects. This point of view is used and developed further in the companion papers \cite{ar, ar2}.

In the case $\E=\K$ the category $\Parity_{(\Bv)}(\cBv,\E)$ coincides with the category $\mathsf{IC}_{(\Bv)}(\cBv,\K)$ of~\S\ref{ss:intro-properties}. In this case Theorem \ref{thm:main} can be deduced from \cite{by}. (It can also be deduced from the more standard Koszul duality of \cite{bgs} together with the Radon transform of \cite{bbm}; see Property \eqref{it:property-Ringel} in~\S\ref{ss:intro-properties}.) Our strategy in the general case will be similar to the one used in \cite{by}, relating our categories to certain categories of Soergel modules. 

The remainder of Section~\ref{sec:main-result} is devoted to various applications of Theorem~\ref{thm:main}.

\subsection{Application to multiplicities}
\label{ss:multiplicities}

In this subsection we consider the case $\E=\F$. The first consequence of Theorem \ref{thm:main} is the following description of composition multiplicities of costandard objects (or equivalently---using Verdier duality---of standard objects) in the category $\Perv_{(B)}(\cB,\F)$.

\begin{thm}
\label{thm:multiplicities}
Assume that $\ell$ is good for $G$.
For any $v,w \in W$ we have
\[
[ \nabla_w : \IC_v ] = (\cP_v : \Delta_w) = (\cT_{v w_0} : \nabla_{w w_0}) = \dim \mathbb{H}^\bullet(\cBv_{w_0 w^{-1}}, \iv_{w_0 w^{-1}}^* \cEv_{w_0 v^{-1}}).
\]
\end{thm}

\begin{proof}
By standard arguments one can assume that $\#\F > 2$. Then
the first equality is the usual reciprocity formula in the highest weight category $\Perv_{(B)}(\cB,\F)$; for instance, see~\cite{bgs}. The second equality is given by \eqref{eqn:multiplicities-T-P}. Finally, the third equality follows from $(2)$ and $(3)$ in Theorem \ref{thm:main}.
\end{proof}

\begin{rmk}
\begin{enumerate}
\item
The matrix $\bigl( [\nabla_w : \IC_v] \bigr)_{v,w \in W}=\bigl( [\Delta_w : \IC_v] \bigr)_{v,w \in W}$ is invertible, and its inverse is the matrix $\bigl( (-1)^{\dim \cB_v} \chi_v(\IC_w) \bigr)_{v,w \in W}$, where 
$\chi_v(\cF)$ is the Euler characteristic of the cohomology of the stalk of $\cF$ at $vB/B$. Hence Theorem \ref{thm:multiplicities} also enables to compute the values of $\chi_v(\IC_w)$ for $v,w \in W$.
\item
Recall from~\S\ref{ss:intro-applications} that the rightmost term in Theorem \ref{thm:multiplicities} can be computed algorithmically. Thanks to this result, the same holds for the other quantities as well.
\item
Computations in low rank (by G.~Williamson and P.~Polo) indicate that Theorem~\ref{thm:multiplicities} might also hold in bad characteristic. Hence Theorem~\ref{thm:main} may hold without any assumption on $\ell$.
\item
One can show that $\dim \mathbb{H}^\bullet(\cBv_{x}, \iv_{x}^* \cEv_{y}) = \dim \mathbb{H}^\bullet(\cBv_{x^{-1}}, \iv_{x^{-1}}^* \cEv_{y^{-1}})$ for any $x,y \in W$. Hence in Theorem~\ref{thm:multiplicities} one can replace the last quantity by $\dim \mathbb{H}^\bullet(\cBv_{w w_0}, \iv_{w w_0}^* \cEv_{v w_0})$.
\end{enumerate}
\end{rmk}

\subsection{Application to Soergel's modular category $\mathcal{O}$}
\label{ss:modular-O}

In this subsection we assume that $\ell$ is bigger than the Coxeter number of $G$. We denote by $G_\F$ a split simply-connected semisimple algebraic $\F$-group whose root system is isomorphic to that of $G$. We choose a maximal torus $T_\F$ in $G_\F$, and denote by $B_\F^-$ the Borel subgroup of $G_\F$ containing $T_\F$ whose roots are the negative roots.

In \cite{soergel}, Soergel defines the ``modular category $\mathcal{O}$'' associated with $G_\F$ as a certain subquotient of the category of finite dimensional (rational) $G_\F$-modules ``around the Steinberg weight.'' To emphasize the difference with the ordinary category $\mathcal{O}$, we denote this category by $\mathcal{O}_{\F}$. It is a highest weight category with weight poset $W$ (endowed with the \emph{inverse} of the Bruhat order). In particular we have simple objects $\{L_w, w \in W\}$ (where, as in \cite{soergel}, $L_w$ is the image in the quotient category of the simple $G_\bk$-module $L((\ell-1)\rho + w\rho)$ of highest weight $(\ell-1)\rho + w\rho$, with $\rho$ the half sum of positive roots), standard objects $\{M_w, w \in W\}$ (where $M_w$ is the image of the Weyl module $V((\ell-1)\rho + w \rho)$) and costandard objects $\{N_w, w \in W\}$ (where $N_w$ is the image of the induced modules $\mathrm{Ind}_{B^-_\F}^{G_\F}((\ell-1)\rho + w \rho)$). We also denote by $P_w$ the projective cover of $L_w$.

The following result is an analogue of a well-known result for the ordinary category $\mathcal{O}$; see e.g.~\cite[Proposition 3.5.2]{bgs}.

\begin{thm}
\label{thm:O-Perv}
Assume that $\ell$ is bigger than the Coxeter number of $G$.
There exists an equivalence of abelian categories
\[
\Psi \colon \mathcal{O}_\F \simto \Perv_{(B)}(\cB,\F)
\]
such that
\[
\Psi(L_w) \cong \IC_{w^{-1} w_0}, \quad \Psi(M_w) \cong \Delta_{w^{-1} w_0}, \quad \Psi(N_w) \cong \nabla_{w^{-1} w_0}, \quad \Psi(P_w) \cong \cP_{w^{-1} w_0}.
\]
\end{thm}

\begin{rmk}
\begin{enumerate}
\item
By construction, one can see the ``multiplicities around the Steinberg weight'' for $G_\F$-modules in the category $\mathcal{O}_\F$. More precisely, for $v,w \in W$ the multiplicity $[\mathrm{Ind}_{B_\F^-}^{G_\F}((\ell-1)\rho + w \rho):L((\ell-1)\rho + v\rho)]$ of the simple $G_\F$-module $L((\ell-1)\rho + v\rho)$ in the induced $G_\F$-module $\mathrm{Ind}_{B_\F^-}^{G_\F}((\ell-1)\rho + w \rho)$ equals the multiplicity $[N_w:L_v]$. It follows from Theorem \ref{thm:O-Perv} that this multiplicity is also the multiplicity of $\IC_{v^{-1} w_0}$ in $\nabla_{w^{-1} w_0}$. Hence our theorem provides multiplicity formulas which are \emph{different} from those of \cite[Theorem 1.2]{soergel}. The relation between these multiplicity formulas is explained by Theorem \ref{thm:multiplicities}.
\item
One can give a more direct proof of Theorem~\ref{thm:O-Perv} using directly Theorem~\ref{thm:equivalence-tilting} below; however to state the latter result requires introducing more notation.
\end{enumerate}
\end{rmk}

\begin{proof}
Recall (see~\S\ref{ss:intro-properties})
that  $\Proj\lh\mathcal{O}_\F$ denotes the additive full subcategory of $\mathcal{O}_\F$ consisting of projective objects. This category is Krull--Schmidt, and its indecomposable objects are the $P_w$'s ($w \in W$). By the main results of \cite{soergel}, there exists a functor
\[
\eta \colon \Parity_{(\Bv)}(\cBv,\F) \to \Proj\lh\mathcal{O}_\F
\]
and an isomorphism of functors $\varepsilon' \colon \eta \circ [1] \simto \eta$ such that:
\begin{enumerate}
\item 
for any $\cEv,{\check \cF}$ in $\Parity_{(\Bv)}(\cBv,\E)$, the functor $\eta$ and the isomorphism $\varepsilon'$ induce an isomorphism
\[
\bigoplus_{n \in \Z} \Hom \bigl( \cEv,{\check \cF}[n] \bigr) \simto \Hom \bigl( \eta(\cEv),\eta({\check \cF}) \bigr);
\]
\item
for any $w \in W$ we have $\eta_\E(\cEv_w) \cong P_w$.
\end{enumerate}
From this and Theorem \ref{thm:main} (in case $\E=\F$) one can easily check that there exists a unique (up to isomorphism) equivalence of categories
\[
\Phi \colon \Proj\lh\mathcal{O}_\F \simto \Tilt_{(B)}(\cB,\F)
\]
which satisfies $\Phi(P_w) \cong \cT_{w^{-1}}$ and which is compatible with $\eta$ and $\nu$ in the natural sense. Composing this equivalence with $\mathsf{R}$ (see~\S\ref{ss:reminder-radon}) we obtain an equivalence
\[
\Phi' \colon \Proj\lh\mathcal{O}_\F \simto \Proj_{(B)}(\cB,\F)
\]
which satisfies $\Phi'(P_w) \cong \cP_{w^{-1} w_0}$.

Consider now the following equivalence:
\[
\Db \mathcal{O}_\F \xleftarrow{\sim} \Kb \bigl( \Proj \lh \mathcal{O}_\F \bigr) \xrightarrow[\sim]{\Phi'} \Kb \bigl( \Proj_{(B)}(\cB,\F) \bigr) \xrightarrow{\sim} \Db \Perv_{(B)}(\cB,\F).
\]
(Here, the left, resp.~right, equivalence is the natural functor, which is an equivalence since $\mathcal{O}_\F$, resp.~$\Perv_{(B)}(\cB,\F)$, has finite global dimension; these properties follow from \cite[Corollary 3.2.2]{bgs}.) Since the standard t-structure on the left-hand side, resp.~right-hand side, can be described in terms of morphisms from projective objects, this equivalence is t-exact, and hence restricts to an exact equivalence
\[
\Psi \colon \mathcal{O}_\F \simto \Perv_{(B)}(\cB,\F).
\]
By construction this equivalence sends $P_w$ to $\cP_{w^{-1} w_0}$, so it also sends $L_w$ (which is the unique simple quotient of $P_w$) to $\IC_{w^{-1} w_0}$ (which is the unique simple quotient of $\cP_{w w_0}$). To prove that $\Psi$ sends $M_w$ to $\Delta_{w^{-1} w_0}$ we use the fact that $M_w$ is the projective cover of $L_w$ in the Serre subcategory of  $\mathcal{O}_\F$ generated by objects $L_v$ with $v \geq w$, and that $\Delta_{w^{-1} w_0}$ is the projective cover of $\IC_{w^{-1} w_0}$ in the Serre subcategory of $\Perv_{(B)}(\cB,\F)$ generated by objects $\IC_{v^{-1} w_0}$ with $v>w$. A similar argument applies to costandard objects.
\end{proof}

\subsection{Application to decomposition numbers}
\label{ss:decomposition}

We come back to our general assumption that $\ell$ is good for $G$.

First we consider projective objects in $\Perv_{(B)}(\cB,\E)$, for $\E=\K$, $\O$, or $\F$. By construction (see \cite[Corollary 2.4.2]{rsw}), for any $w \in W$ we have
$\F(\cP_w^\O) \cong \cP_w^\F$.
On the other hand the perverse sheaf $\K(\cP_w^\O)$ is projective in $\Perv_{(B)}(\cB,\K)$, so there exist coefficients $p_{v,w} \in \Z_{\geq 0}$ (for $v,w \in W$) such that
\[
\K(\cP_w^\O) \cong \bigoplus_{v \in W} \bigl( \cP_v^\K \bigr)^{\oplus p_{v,w}}.
\]
We denote by $\mathfrak{P}$ the matrix $(p_{v,w})_{v,w \in W}$, and by $\mathfrak{P}'$ the matrix $(p_{v w_0, w w_0})_{v,w \in W}$.

Similarly, for any $w \in W$ we have
$\F(\cT_w^\O) \cong \cT_w^\F$
(see Proposition~\ref{prop:properties-tilting})
and there exist coefficients $t_{v,w} \in \Z_{\geq 0}$ such that
\[
\K(\cT_w^\O) \cong \bigoplus_{v \in W} \bigl( \cT_v^\K \bigr)^{\oplus t_{v,w}}
\]
We denote by $\mathfrak{T}$ the matrix with coefficients $t_{v,w}$.

Likewise, by \cite[Proposition 2.40]{jmw}, for $w \in W$ we have
\begin{equation}
\label{eqn:F-parity}
\F(\cEv_w^\O) \cong \cEv_w^\F
\end{equation}
and there exist coefficients ${\check e}_{v,w}^i \in \Z_{\geq 0}$ such that
\[
\K(\cEv_w^\O) \cong \bigoplus_{v \in W, i \in \Z} \bigl(\cEv_v^\K [i] \bigr)^{\oplus {\check e}^i_{v,w}}.
\]
We denote by $\check{\mathfrak{E}}$ the matrix whose $(v,w)$-coefficient is $\sum_i {\check e}_{v^{-1},w^{-1}}^i$.

Finally we consider simple objects in the category $\Perv_{(B)}(\cB,\E)$. It follows from the definitions that we have
$\K(\IC_w^\O) \cong \IC_w^\K$,
and moreover $\F(\IC_w^\O)$ is a perverse sheaf. (This follows from the fact that $\IC_w^\O$ is torsion-free, since the functor $i_{w!*}$ preserves injections, see~\cite[Proposition~2.27]{juteau}.) We denote by $\mathfrak{I}$ the matrix $\bigl( [\F(\IC_w^\O) : \IC_v^\F] \bigr)_{v,w \in W}$, and by ${}^t \mathfrak{I}$ its transpose.

\begin{thm}
\label{thm:decomposition}
Assume that $\ell$ is good for $G$.
We have equalities
\[
\mathfrak{P}' = \mathfrak{T} = \check{\mathfrak{E}}, \qquad {}^t \mathfrak{I}=\mathfrak{P}.
\]
\end{thm}

\begin{proof}
By standard arguments one can assume that $\#\F > 2$.
Then the equality $\mathfrak{P}'=\mathfrak{T}$ follows from the existence the Radon transform over $\O$ and $\K$, and their compatibility with extension of scalars (see~\S\ref{ss:reminder-radon}). The equality $\mathfrak{T} = \check{\mathfrak{E}}$ follows from the existence of the functors $\nu$ over $\O$ and $\K$, and their compatibility with extension of scalars (see Theorem \ref{thm:main}).

The last equality is an instance of Brauer reciprocity. More precisely, consider the $\O$-module
\[
\Hom(\cP_w^\O, \IC_v^\O).
\]
By Lemma \ref{lem:Hom-proj-F} below, this $\O$-module is free, and we have
\[
\F \otimes_\O \Hom(\cP_w^\O, \IC_v^\O) \cong \Hom \bigl( \F(\cP_w^\O), \F(\IC_v^\O) \bigr) = \Hom \bigl( \cP_w^\F, \F(\IC_v^\O) \bigr).
\]
We deduce that the rank of our $\O$-module is $[\F(\IC_v^\O) : \IC_w^\F]$. On the other hand we have
\[
\K \otimes_\O \Hom(\cP_w^\O, \IC_v^\O) \cong \Hom \bigl( \K(\cP_w^\O), \K(\IC_v^\O) \bigr) \cong \Hom \bigl( \K(\cP_w^\O), \IC_v^\K \bigr).
\]
Hence the rank of our $\O$-module is also equal to $p_{v,w}$. We thereby obtain
\[
p_{v,w} = [\F(\IC_v^\O) : \IC_w^\F],
\]
and our last equality follows.
\end{proof}

\section{Soergel modules}
\label{sec:soergel-modules}

In this section, we begin laying the foundations for the proof of Theorem~\ref{thm:main}. In particular we define the ``Soergel modules'' which will play a key role in our arguments. (It will turn out that these modules over the coinvariant algebra are nothing but the global cohomology of parity sheaves on $\cBv$, see Corollary~\ref{cor:H-BS}, just as characteristic-$0$ Soergel modules are the global cohomology of semisimple complexes on $\cBv$.)

Note that the categories appearing in that result depend only on the root systems of $G$ and $\Gv$, so we may restrict our attention to some family of reductive groups that covers all possible root systems.

\subsection{Additional notation and conventions}
\label{ss:new-notation}

For the remainder of the paper, we assume that $G$ is a product of groups isomorphic either to $\GL_n(\C)$ or else to a simple group (of adjoint type) not of type $\mathbf{A}$.  This assumption implies that the Langlands dual group $\Gv$ is a product of groups isomorphic either to $\GL_n(\C)$ or else to a simply-connected quasi-simple group not of type $\mathbf{A}$. We also assume that $\ell$ is good for $G$.

Let $\bfY := X_*(T) = X^*(\Tv)$.  We set $\fh:=\E \otimes_\Z \bfY$ and $S:=\mathrm{S}_{\E}(\fh)$ (i.e.~the symmetric algebra of the free $\E$-module $\fh$), considered as a graded $\E$-algebra where $\fh$ is in degree $2$. We endow $S$ with the natural action of $W$. We denote by $S^W_+ \subset S$ the ideal generated by homogeneous elements in $S^W$ (the $W$-invariants in $S$) of positive degree. The graded $\E$-algebra $C:=S/S^W_+$ (usually called the coinvariant algebra) will play a major role in the paper.

If $s$ is a simple reflection, we denote by $C_s \subset C$ the image of the $s$-invariants $S^s \subset S$ in $C$. Note that unless $\E$ is a field of characteristic $2$ (in which case $G$ is necessarily a product of general linear groups), $C_s$ coincides with the $s$-invariants in $C$ (see e.g.~the proof of Lemma \ref{lem:invariants} below).

\subsection{Coinvariants and extension of scalars}

The following lemma collects some technical results (contained in or easily deduced from~\cite{demazure}) that will be needed later.

\begin{lem}
\label{lem:invariants}
\begin{enumerate}
\item
The natural morphisms
\[
\F \otimes_\O (S_\O)^W \to (S_\F)^W \qquad \text{and} \qquad \K \otimes_\O (S_\O)^W \to (S_\K)^W
\]
are isomorphisms.
\item
The natural morphisms
\[
\F \otimes_\O C_\O \to C_\F \qquad \text{and} \qquad \K \otimes_\O C_\O \to C_\K
\]
are isomorphisms.
\item
The $(S_\E)^W$-module $S_\E$ is free of rank $\# W$, and $C_\E$ is $\E$-free of rank $\# W$.
\item
For any simple reflection $s$, the natural morphisms
\[
\F \otimes_\O (S_\O)^s \to (S_\F)^s \qquad \text{and} \qquad \K \otimes_\O (S_\O)^s \to (S_\K)^s
\]
are isomorphisms.
\item
The $C_s^\E$-module $C_\E$ is free of rank $2$.
\end{enumerate}
\end{lem}

\begin{proof}
Let us begin with $(1)$.
Set $S_\Z = \mathrm{S}_\Z(\bfY)$. Then under our assumptions, by \cite[Corollaire on p.~296]{demazure}, the natural morphism
\[
\E \otimes_\Z (S_\Z)^W \to (S_\E)^W
\]
is an isomorphism for $\E=\K,\O$ or $\F$, which proves the claim.

$(2)$ is a direct consequence of $(1)$: in fact one can define $C_\Z$ in the natural way, and the morphism $\E \otimes_\Z C_\Z \to C_\E$ is an isomorphism for $\E=\K,\O$ or $\F$.

$(3)$ follows from \cite[Th{\'e}or{\`e}me 2(c)]{demazure} and the proof of $(1)$.

Let us finally consider $(4)$ and $(5)$. These claims are clear in case $G=\GL_n(\C)$. If $G$ is a simple group not of type $\mathbf{A}$ then since $\ell \neq 2$ we have
$\fh = \ker(s-1) \oplus \ker(s+1)$,
and $\ker(s+1)$ is free of rank $1$; hence the claims are clear also. The general case follows, since by assumption $G$ is a product of groups either isomorphic to $G=\GL_n(\C)$ or to a simple group not of type $\mathbf{A}$.
\end{proof}

\subsection{Bott--Samelson modules}
\label{ss:BS-modules}

For any sequence $\us=(s_1, \ldots, s_i)$ of simple reflections of $W$ we will consider the graded $C$-module
\[
\mathsf{BS}^\gr(\us) := S \otimes_{S^{s_i}} S \otimes_{S^{s_{i-1}}} \otimes \cdots \otimes_{S^{s_1}} \E \la i \ra.
\]
Here, $\E$ is considered as a graded $S$-module via the canonical identification $S/\fh S \cong \E$ (and the $S^{s_1}$-module structure is obtained by restriction). Note that the $S$-module structure on $\E$ factors through an action of $C$, and that we have
\[
\mathsf{BS}^\gr(\us) := C \otimes_{C_{s_i}} C \otimes_{C_{s_{i-1}}} \otimes \cdots \otimes_{C_{s_1}} \E \la i \ra.
\]
Note also the inversion of the order of the simple reflections.

We denote by $\cS^\gr$ the smallest strictly full subcategory of the category of graded $C$-modules which contains the graded module $\mathsf{BS}^\gr(\us)$ for any sequence $\us$ of simple reflections and which is stable under shifts $\la 1 \ra$, direct sums and direct summands. We will call objects of $\cS^\gr$ \emph{Soergel modules}. This is justified by the fact 
that these objects can be obtained from what are usually called \emph{Soergel bimodules} by killing one of the two actions of $S$.\footnote{Following a referee's suggestion, let us be more specific about this claim. In fact, Theorem~\ref{thm:equivalence-parity} below can be generalized to the case of \emph{$\Bv$-equivariant} parity sheaves, if one replaces $\cS^\gr_\E$ by the appropriate category of Soergel bimodules. In particular, it follows that (graded) indecomposable Soergel bimodules are classified by $W \times \Z$ in the natural way. Then, using the fact that if $\cF$ is a $\Bv$-equivariant parity sheaf on $\cBv$, the natural morphism $\E \otimes_{\mathbb{H}_{\Bv}^\bullet(\mathrm{pt}; \E)} \mathbb{H}_{\Bv}^\bullet(\cBv, \cF) \to \mathbb{H}^\bullet(\cBv, \cF)$ is an isomorphism (as follows from~\cite[Proposition~2.6]{jmw}), one can check that the functor $\E \otimes_{S} (-)$ sends the indecomposable bimodule associated with $(w,0)$ to the indecomposable Soergel module $D_w^\gr$ of~\S\ref{ss:modular-BS}.}

Note also that by Lemma \ref{lem:invariants}$(4)$ we have
\[
\F \otimes_\O \mathsf{BS}_\O^\gr(\us) \cong \mathsf{BS}_{\F}^\gr(\us).
\]
We deduce that the functor $\F \otimes_\O (-)$ induces a functor
$\F(-) \colon \cS_\O^\gr \to \cS_\F^\gr$. Similar remarks show that the functor $\K \otimes_\O(-)$ induces a functor $\K(-) \colon \cS_\O^\gr \to \cS_\K^\gr$.

We will also denote by $\cS$ the smallest strictly full subcategory of the category of (ungraded) $C$-modules which contains the module
$\mathsf{BS}(\us) := \For \bigl( \mathsf{BS}^\gr(\us) \bigr)$ for any sequence $\us=(s_1, \ldots, s_i)$ of simple reflections and which is stable under direct sums and direct summands. As above, the functor $\F \otimes_\O (-)$ induces a functor
$\F(-) \colon \cS_\O \to \cS_\F$, and the functor $\K \otimes_\O(-)$ induces a functor $\K(-) \colon \cS_\O \to \cS_\K$.

Note that the categories $\cS^\gr$ and $\cS$ are Krull--Schmidt. This is clear if $\E=\K$ or $\F$; for the case $\E=\O$, see the arguments in the proof of \cite[Lemma 2.1.6]{rsw}.

We also have the functor $\For \colon \cS^\gr \to \cS$ which forgets the grading. We will generalize the following result to the case $\E=\O$ later (see~\S\ref{ss:modular-BS}).

\begin{lem}
\label{lem:v-indecomposable}
If $\E=\K$ or $\F$ and if $D$ is an indecomposable object of $\cS^\gr$, then $\For(D)$ is indecomposable in $\cS$.
\end{lem}

\begin{proof}
This follows from general results on graded modules over finite-dimensional $\E$-algebras; see \cite[Theorem 3.2]{gg}.
\end{proof}

\section{Parity sheaves and the functor $\mathbb{H}$}
\label{sec:hh}

In this section, we relate parity sheaves on $\cBv$ to Soergel modules.
The results of this subsection are well known (and due to Soergel) in case $\E=\K$, see~\cite{soergel-kategorie}. In case $\E=\F$ and $\ell$ is bigger than the Coxeter number of $\Gv$, they are also proved (using different arguments) in~\cite{soergel}. (Our notation follows Soergel's notation in~\cite{soergel}.)

The conventions of~\S\ref{ss:new-notation} remain in effect.

\subsection{Cohomology of $\cBv$}

The following result is well known, but we couldn't find a reference treating this question under our assumptions. For completeness, we explain how it can be deduced from the existing literature.

\begin{prop}
\label{prop:cohomology-G/B}
There exists a natural isomorphism of graded algebras
\[
C_\E \simto \mathbb{H}^\bullet(\cBv;\E).
\]
\end{prop}

\begin{proof}
Recall the algebras $S_\Z$ and $C_\Z$ considered in the proof of Lemma~\ref{lem:invariants}.
The Chern character provides a natural algebra morphism $S_\Z \to A(\cBv)$ (where $A(\cBv)$ is the Chow ring of $\cBv$), which factors through a morphism $C_\Z \to A(\cBv)$; see e.g.~\cite[Corollaire 2 and \S8]{demazure}. By~\cite[Th{\'e}or{\`e}me 2(a) and \S8]{demazure}, under our assumptions this morphism induces a surjection
\[
\E \otimes_\Z C_\Z \twoheadrightarrow \E \otimes_\Z A(\cBv).
\]
By the proof of Lemma \ref{lem:invariants}(2), we have $\E \otimes_\Z C_\Z \cong C_\E$. On the other hand,
by \cite[Lemma 4.1]{ahjr} there exists a canonical isomorphism $\E \otimes_\Z A(\cBv) \simto \mathbb{H}^\bullet(\cBv;\E)$. Hence we obtain a surjection
\[
C_\E \twoheadrightarrow \mathbb{H}^\bullet(\cBv;\E).
\]
Since both $\E$-modules are free of rank $\# W$ by Lemma \ref{lem:invariants}(3) and the Bruhat decomposition, this surjection must be an isomorphism.
\end{proof}

If $s$ is a simple reflection, we denote by $\Pv^s \subset \Gv$ the minimal standard parabolic subgroup associated with $s$, and set $\scPv^s:=\Gv/\Pv^s$. We will denote by ${\check \pi}^s \colon \cBv \to \scPv^s$ the natural projection.

\begin{cor}
\label{cor:cohomology-G/P}
The morphism induced by $({\check \pi}^s)^*$ provides an isomorphism
\[
\mathbb{H}^\bullet(\scPv^s;\E) \cong C_s.
\]
\end{cor}

\begin{proof}
Assume first that $\E$ is not a field of characteristic $2$.
In this case it is well known (see e.g.~\cite[\S 3.2 and references therein]{soergel-kategorie}) that the isomorphism of Proposition \ref{prop:cohomology-G/B} commutes with the natural actions of $W$, and that the morphism
\[
({\check \pi}^s)^* \colon \mathbb{H}^\bullet(\scPv^s;\E) \to \mathbb{H}^\bullet(\cBv;\E)
\]
is injective and identifies the left-hand side with the $s$-invariants in the right-hand side, which proves the claim.

If $\E$ is a field of characteristic $2$, then $\Gv$ is a product of general linear groups, and the result can be deduced from the case $\E=\O$.
\end{proof}

\subsection{The functor $\mathbb{H}$ and Bott--Samelson parity sheaves}

Using the isomorphism of Proposition \ref{prop:cohomology-G/B} one can consider the functor
\[
\mathbb{H}^\bullet(\cBv, -) \colon \Parity_{(\Bv)}(\cBv, \E) \to C\lh\gmod.
\]
(Note that it would be more natural to consider this functor as taking values in the category of graded \emph{right} $C$-modules. However, since $C$ is commutative, right and left $C$-modules are the same.)

For any simple reflection $s$, we define the functor
\[
\vartheta_s := ({\check \pi}^{s})^* ({\check \pi}^s)_* \colon \Db_{(\Bv)}(\cBv, \E) \to \Db_{(\Bv)}(\cBv, \E).
\]

\begin{lem}
\label{lem:cohomology-theta}
For any ${\check \cF}$ in $\Db_{(\Bv)}(\cBv, \E)$ and any simple reflection $s$, there exists a functorial isomorphism of graded $C$-modules
\[
\mathbb{H}^\bullet ( \cBv, \vartheta_s {\check \cF} ) \cong C \otimes_{C_s} \mathbb{H}^\bullet(\cBv, {\check \cF}).
\]
\end{lem}

\begin{proof}
This claim follows from \cite[Proposition 4.1.1]{soergel}, Proposition~\ref{prop:cohomology-G/B} and Corollary \ref{cor:cohomology-G/P}. (In the proof of \cite[Proposition 4.1.1]{soergel} the author assumes that the coefficients are a field of characteristic $\neq 2$, but this  assumption is not needed: all that one needs is the existence of an isomorphism
$({\check \pi}^s)_* \underline{\E}_{\cBv} \cong \underline{\E}_{\scPv^s} \oplus \underline{\E}_{\scPv^s}[-2]$,
which follows from the facts that
\[
\mathcal{H}^n \bigl( ({\check \pi}^s)_* \underline{\E}_{\cBv} \bigr) = \begin{cases}
\underline{\E}_{\scPv^s} & \text{if $n \in \{0,2\}$;} \\
0 & \text{otherwise}
\end{cases}
\]
and that $\mathbb{H}^3(\scPv^s;\E)=0$.)
\end{proof}

It is well known (see \cite[\S 4.1]{jmw}) that the functors $({\check \pi}^{s})^*$ and $({\check \pi}^{s})_*$ send parity complexes to parity complexes. We deduce that the functor $\vartheta_s$ restricts to a functor
\[
\vartheta_s \colon \Parity_{(\Bv)}(\cBv, \E) \to \Parity_{(\Bv)}(\cBv, \E).
\]
For each sequence $\us=(s_1, \ldots, s_i)$ of simple reflections we will consider the ``Bott--Samelson'' parity complex
\[
\cEv(\us) := \vartheta_{s_i} \cdots \vartheta_{s_1} \underline{\E}_{\cBv_e} [i].
\]
(Here, $\underline{\E}_{\cBv_e}$ is the skyscraper sheaf at the origin $\cBv_e=\{\Bv/\Bv\}$.)

The following result is an immediate consequence of Lemma \ref{lem:cohomology-theta}.

\begin{cor}
\label{cor:H-BS}
For any sequence $\us$ of simple reflections, there exists a canonical isomorphism of graded $C$-modules
\[
\mathbb{H}^\bullet \bigl( \cBv,\cEv(\us) \bigr) \cong \mathsf{BS}^\gr(\us).
\]
\end{cor}

\subsection{Equivalence}

The main result of this section is the following. It follows rather easily from known results, and is a direct analogue of Soergel's results in~\cite{soergel-kategorie} which motivated the general study of Soergel bimodules.

\begin{thm}
\label{thm:equivalence-parity}
The functor $\mathbb{H}^\bullet(\cBv,-) \colon \Parity_{(\Bv)}(\cBv,\E) \to C\lh\gmod$ is fully faithful, and induces an equivalence of categories
\[
\mathbb{H} \colon \Parity_{(\Bv)}(\cBv,\E) \simto \cS^\gr_\E
\]
which satisfies $\mathbb{H} \circ [1] = \la 1 \ra \circ \mathbb{H}$. Moreover, the following diagram commutes (up to natural equivalence):
\[
\xymatrix@C=1.5cm{
\Parity_{(\Bv)}(\cBv,\K) \ar[d]_-{\mathbb{H}_\K}^-{\wr} & \Parity_{(\Bv)}(\cBv,\O) \ar[r]^-{\F(-)} \ar[l]_-{\K(-)} \ar[d]^-{\mathbb{H}_\O}_-{\wr} & \Parity_{(\Bv)}(\cBv,\F) \ar[d]^-{\mathbb{H}_\F}_-{\wr} \\
\cS^\gr_\K & \cS^\gr_\O \ar[r]^-{\F(-)} \ar[l]_-{\K(-)} & \cS^\gr_\F.
}
\]
\end{thm}

\begin{proof}
The fact that $\mathbb{H}^\bullet(\cBv,-)$ is fully faithful is proved in \cite[Theorem 4.1]{arider}. (In \emph{loc.}~\emph{cit}.~the authors assume that the ring of coefficients is a field, but the same proof applies also to the case $\E=\O$.) To identify the essential image of this functor we note that the category $\Parity_{(\Bv)}(\cBv,\E)$ is generated (under taking direct summands, direct sums and shifts) by the objects $\cEv(\us)$; see \cite[\S 4.1]{jmw}. Then the result follows from Corollary \ref{cor:H-BS} and the definition of $\cS^\gr_\E$.

The compatibility of $\mathbb{H}$ with the functors $\K(-)$ is obvious; the compatibility with $\F(-)$ follows from \cite[Eq.~(2.13)]{jmw}. (Note that the constant sheaf $\underline{\E}_{\cBv}$ is a parity complex.)
\end{proof}

\subsection{Consequences for the structure of $\cS$ and $\cS^\gr$}
\label{ss:modular-BS}

Recall from~\S\ref{ss:background-parity} that the indecomposable objects in the Krull--Schmidt category $\Parity_{(\Bv)}(\cBv,\E)$ are parame\-trized by $W \times \Z$. It follows from Theorem \ref{thm:equivalence-parity} that the same is true for the category $\cS^\gr$. More precisely, for any $w \in W$ we define the indecomposable object
\[
D_w^\gr := \mathbb{H}(\cEv_w).
\]
If $w=s_1 \cdots s_{\ell(w)}$ is any reduced expression for $w$, then $D_w^\gr$ is characterized by the fact that it appears as a direct summand of $\mathsf{BS}^\gr(s_1, \ldots, s_{\ell(w)})$, and does not appear as a direct summand of $\mathsf{BS}^\gr(\us) \la n \ra$ for any $n \in \Z$ and any sequence $\us$ of simple reflections of length strictly less than $\ell(w)$. Moreover, any indecomposable object of $\cS^\gr_\E$ is isomorphic to $D_w^\gr \la n \ra$ for some unique $w \in W$ and $n \in \Z$.

Note also that by \eqref{eqn:F-parity} and the commutativity of the right-hand side of the diagram in Theorem \ref{thm:equivalence-parity} we have
\begin{equation}
\label{eqn:F-D}
\F(D_{w,\O}^\gr) \cong D_{w,\F}^\gr.
\end{equation}
(The analogous result for $\K$ does not hold in general: the decomposition of $\K(D_{w,\O}^\gr)$ is governed by the integers ${\check e}^i_{v,w}$ of~\S\ref{ss:decomposition}.)

The following technical result, which we obtain as a corollary of Theorem \ref{thm:equivalence-parity}, will be needed later on. (Note that this result is not obvious from definitions, and that our proof uses geometry, hence cannot be adapted to a general Coxeter group.)

\begin{cor}
\label{cor:morphisms-BS}
For any $D$ and $D'$ in $\cS_\O$, the natural morphism
\[
\F \otimes_\O \Hom_{C_\O} \bigl( D, D' \bigr) \to \Hom_{C_\F} \bigl( \F(D), \F(D') \bigr)
\]
is an isomorphism.
\end{cor}

\begin{proof}
It is enough to prove the claim when $D=\mathsf{BS}_\O(\us)$ and $D'=\mathsf{BS}_\O(\ut)$ for some sequences of simple reflections $\us$ and $\ut$. In this case,
using equivalences $\mathbb{H}_\O$ and $\mathbb{H}_\F$ (and their compatibility with functors $\F(-)$) this claim reduces to the claim that the natural morphism
\[
\F \otimes_\O \Hom(\cEv_\O(\us), \cEv_\O(\ut)[i]) \to \Hom(\cEv_\F(\us), \cEv_\F(\ut)[i])
\]
is an isomorphism for any $i \in \Z$, and this follows from \cite[Eq.~(2.13)]{jmw}.
\end{proof}

For any $w \in W$, we define the object
\[
D_w := \For(D_w^\gr)
\]
of $\cS$. From~\eqref{eqn:F-D} we deduce that
\begin{equation}
\label{eqn:F-D-2}
\F(D_{w,\O}) \cong D_{w,\F}.
\end{equation}

\begin{cor}
\label{cor:D_w-indecomposable}
\begin{enumerate}
\item
For any $w \in W$, the object $D_{w}$ is indecomposable.
\item
The objects $\{D_w, \ w \in W\}$ form a complete set of pairwise nonisomorphic indecomposable objects in the Krull--Schmidt category $\cS$.
\end{enumerate}
\end{cor}

\begin{proof}
(1) If $\E=\F$ or $\K$, $D_w$ is indecomposable by Lemma \ref{lem:v-indecomposable}. If $\E=\O$,
by \eqref{eqn:F-D-2} we have $\F(D_{w,\O})=D_{w,\F}$; in particular, this object is indecomposable. It follows that $D_{w,\O}$ is itself indecomposable.

(2) It is enough to prove that if $v \neq w$ then $D_v$ is not isomorphic to $D_w$. Using~\eqref{eqn:F-D-2}, it is enough to prove this property when $\E=\K$ or $\F$. In this case, the claim follows from~(1) and~\cite[Lemma 2.5.3]{bgs}.
\end{proof}

In complement to Corollary~\ref{cor:D_w-indecomposable}, let us note that if $w=s_1 \cdots s_{\ell(w)}$ is any reduced expression for $w$, then $D_w$ is characterized by the fact that it appears as a direct summand of $\mathsf{BS}(s_1, \ldots, s_{\ell(w)})$, and does not appear as a direct summand of $\mathsf{BS}(\us)$ for any sequence $\us$ of simple reflections of length strictly less than $\ell(w)$.

\begin{rmk}
\label{rmk:Dw0}
When $w=w_0$, the variety $\overline{\cBv_{w_0}}=\cBv$ is smooth, so that we have $\cEv_{w_0}=\underline{\E}_{\cBv}[\dim \cBv]$. It follows that $D_{w_0} = \mathbb{H}^\bullet(\cBv;\E) = C$.
\end{rmk}

\section{Tilting perverse sheaves and the functor $\mathbb{V}$}
\label{sec:tilting-V}

In this section, we relate tilting sheaves on $\cB$ to Soergel modules through a ``functor $\mathbb{V}$.'' As in characteristic $0$ this functor is given by $\Hom(\cP_e, -)$ for an appropriate object $\cP_e$ in $\Perv_{(B)}(\cB,\E)$, where the morphism $C \to \End(\cP_e)$ is obtained by ``taking the logarithm of the monodromy.'' In~\S\ref{ss:monodromy} we make sense of this ``logarithm.'' The rest of the argument proceeds by reduction to the (known) case of coefficients $\K$, by carefully developing the constructions in the three settings of coefficients $\O$, $\F$ and $\K$, and exploiting, at each step, the most favorable case to deduce the other ones via the functors $\F(-)$ and $\K(-)$.

The conventions of~\S\ref{ss:new-notation} remain in effect. We also assume from now on that $\#\F > 2$.

\subsection{Statement}

The main result of this section is the following.

\begin{thm}
\label{thm:equivalence-tilting}
There exist equivalences of additive categories 
\[
\mathbb{V}_{\E} \colon \Tilt_{(B)}(\cB,\E) \simto \cS_\E,
\]
such the following diagram commutes (up to isomorphisms of functors):
\[
\xymatrix@C=1.5cm{
\Tilt_{(B)}(\cB,\K) \ar[d]_-{\mathbb{V}_\K}^-{\wr} & \Tilt_{(B)}(\cB,\O) \ar[r]^-{\F(-)} \ar[l]_-{\K(-)} \ar[d]^-{\mathbb{V}_\O}_-{\wr} & \Tilt_{(B)}(\cB,\F) \ar[d]^-{\mathbb{V}_\F}_-{\wr} \\
\cS_\K & \cS_\O \ar[r]^-{\F(-)} \ar[l]_-{\K(-)} & \cS_\F.
}
\]
\end{thm}

The proof of Theorem \ref{thm:equivalence-tilting} requires a change in setting.  Choose a prime number $p \ne \ell$ such that there exists a primitive $p$-th root of unity in $\F$ (which is possible under our assumption that $\# \F >2$).  In this section only, we replace the \emph{complex} algebraic group $G$ by the reductive algebraic group over $\Fpb$ with the same root datum.  Likewise, $T \subset B$ and $\cB$ are taken to be defined over $\Fpb$, and $\Db_{(B)}(\cB,\E)$ and $\Tilt_{(B)}(\cB,\E)$ are subcategories of the \'etale derived category of $\cB$.  So far, there is no harm in making this change: it is well known that this ``new'' $\Db_{(B)}(\cB,\E)$ (in the \'etale setting over $\Fpb$) is equivalent to the ``old'' one (in the classical topology over $\C$).  See~\cite[Remark~7.1.4(ii)]{rsw} for details.

The full constructible derived category $\Dbc(\cB,\E)$, however, is rather different in the \'etale and classical settings.  In the \'etale setting, it contains ``Whittaker-type'' objects~\cite{bbrm}, which will be a crucial tool in the proof.

\subsection{Preliminary results}
\label{ss:preliminary}

In this subsection we collect a few preliminary lemmas that will be needed in the section.

\begin{lem}
\label{lem:Hom-proj-F}
Let $\cP, \cF \in \Perv_{(B)}(\cB,\O)$, and assume that $\cP$ is projective. (In particular, this implies that $\F(\cP)$ is perverse; see \cite[Proposition 2.4.1]{rsw}.)
\begin{enumerate}
\item
If $\F(\cF)$ is perverse, then the $\O$-module $\Hom(\cP,\cF)$ is free, and the natural morphism
\[
\F \otimes_\O \Hom(\cP,\cF) \to \Hom \bigl( \F(\cP),\F(\cF) \bigr)
\]
is an isomorphism.
\item If $\F(\cP) \cong \F(\cF)$, then $\cP \cong \cF$.
\end{enumerate}
\end{lem}

\begin{proof}
(1) follows from the argument of the proof of \cite[Proposition 2.4.1(ii)]{rsw}: in fact we have
\[
\F \lotimes_\O R\Hom(\cP,\cF) \cong R\Hom(\F(\cP), \F(\cF)).
\]
Now the left-hand side is concentrated in nonpositive degrees, and the right-hand side in nonnegative degrees (since $\F(\cP)$ and $\F(\cF)$ are perverse). Hence both of them are concentrated in degree $0$, which implies our claim.

To prove $(2)$, we observe that $\F \otimes_\O \Hom(\cP,\cF) \cong \Hom( \F(\cP),\F(\cF))$ by $(1)$. Choose some isomorphism $g \colon \F(\cP) \simto \F(\cF)$, and let $f \colon \cP \to \cF$ be a morphism such that $\F(f)=g$. Then the cone of $f$ is annihilated by the functor $\F(-)$ (since $\F(f)$ is an isomorphism), so it is zero (see \cite[Lemma 2.2.1]{rsw}). In other words, $f$ is an isomorphism.
\end{proof}

A standard construction (see e.g.~\cite[\S 3.2]{by}) defines a ``convolution product''
\[
(-) \star^B (-) \colon \Db_{(B)}(\cB,\E) \times \Db_B(\cB,\E) \to \Db_{(B)}(\cB,\E).
\]
(Here, $\Db_B(\cB,\E)$ is the $B$-equivariant constructible derived category in the sense of Bernstein--Lunts~\cite{bl}.\footnote{In~\cite{bl}, Bernstein and Lunts work in a topological setting. However their constructions can be adapted to the algebraic setting (in particular to the setting of {\'e}tale derived categories), see e.g.~\cite[\S 3.1]{weidner-phd} for a detailed treatment.}) The objects $\Delta_w$, $w \in W$ have obvious analogues in $\Db_B(\cB,\E)$, which we denote the same way. The following lemma is well known; see~\cite[\S 2.2]{bbm} or~\cite[Lemma 3.2.2]{by}.

\begin{lem}
\label{lem:convolution-Delta-nabla}
If $v,w \in W$ and $\ell(vw)=\ell(v)+\ell(w)$ then we have
\[
\Delta_v \star^B \Delta_w \cong \Delta_{vw}, \qquad \nabla_v \star^B \nabla_w \cong \nabla_{vw}.
\]
\end{lem}

We finish with another easy lemma, which is probably well known too.

\begin{lem}
\label{lem:orthogonal}
Let $V$ be a finite-rank free $\O$-module, endowed with a perfect pairing\footnote{Here we mean a symmetric bilinear form $\kappa \colon V \times V \to \O$ which induces an isomorphism $V \simto \Hom_{\O}(V,\O)$.} $\kappa$, and let $M \subset V$ be a submodule such that the restriction of $\kappa$ to $M$ is also nondegenerate. If $M^\bot:=\{v \in V \mid \forall m \in M, \ \kappa(v,m)=0\}$ is the orthogonal to $M$, then the natural morphism
$M \oplus M^\bot \to V$
is an isomorphism.
\end{lem}

\begin{proof}
Since our $\O$-modules are free,
it suffices to prove that the image of our morphism under the functor $\F \otimes_\O (-)$ is an isomorphism. Now we observe that the natural map $\F \otimes_\O M \to \F \otimes_\O V$ is injective (otherwise the restriction of $\kappa$ to $M$ could not be nondegenerate), that the restriction of the bilinear form $\kappa_\F$ on $\F \otimes_\O V$ induced by $\kappa$ to the image of this morphism is nondegenerate, and finally that the natural morphism $\F \otimes_\O M^\bot \to \F \otimes_\O V$ is injective and identifies $\F \otimes_\O M^\bot$ with the orthogonal complement of $\F \otimes_\O M$ with respect to $\kappa_\F$. (Indeed, injectivity is easy to check. It is also clear that this morphism factors through $(\F \otimes_\O M)^\bot$; we conclude by a dimension argument.) Hence our claim follows from the analogous claim in the case of $\F$-vector spaces, which is standard.
\end{proof}

\subsection{Logarithm of the monodromy}
\label{ss:monodromy}

Theorem \ref{thm:equivalence-tilting} is proved in case $\E=\K$ in \cite{by}. The main tool of this proof is the ``logarithm of monodromy'' construction. In our modular (or integral) setting this construction cannot be performed directly. In this subsection we introduce a replacement which is available in our context.

We denote by $\E \bfY$ the group algebra of $\bfY=X_*(T)$ (over $\E$). We will also consider the completion
\[
\widehat{\E \bfY}, \qquad \text{resp.} \qquad \hS,
\]
of $\E\bfY$, resp.~$S$, for the topology associated with the ideal $\mathfrak{m}$ generated by elements of the form $y - 1$ for $y \in \bfY$, resp.~the ideal generated by $\fh$. Each of these algebras is endowed with a natural action of $W$.

The following result will be crucial in this section. (But the details of the proof will not be needed.)

\begin{prop}
\label{prop:logarithm}
There is a continuous, $W$-equivariant algebra isomorphism
\[
\hS \simto \widehat{\E \bfY}
\]
which identifies $\fh \cdot \hS$ with $\mathfrak{m} \cdot \widehat{\E \bfY}$.
\end{prop}

\begin{rmk}
\label{rmk:monodromy-BY}
If $\E=\K$ there is a natural choice for the isomorphism of Proposition \ref{prop:logarithm}, namely the logarithm. This is the choice made (from a different point of view) in \cite{by}. However, for us it will be more convenient to choose such isomorphisms over $\O$, $\K$, and $\F$ which are compatible (in the obvious sense).
\end{rmk}

To prove Proposition \ref{prop:logarithm} we adapt an argument from \cite{br}, which itself relies on a technical result from \cite{ss}. To explain this we need some preparation.  Observe first that it is enough to prove the proposition in the case where the root system of $G$ is irreducible (in addition to satisfying the assumptions of~\S\ref{ss:new-notation}).  Therefore, from now until the end of the proof of Proposition~\ref{prop:logarithm}, we assume that $G$ is either $\GL_n(\Fpb)$ or else a simple group (of adjoint type) not of type $\mathbf{A}$.

If $G=\GL_n(\Fpb)$, we set $\widetilde{G}:=\GL_n(\E)$, and we let $\widetilde{T} \subset \widetilde{G}$ denote the maximal torus of diagonal matrices. Otherwise, we choose a split simply connected semisimple $\E$-group $\widetilde{G}$, a maximal torus $\widetilde{T} \subset \widetilde{G}$, and an isomorphism $\bfY \cong X^*(\widetilde{T})$ which identifies the coroot system of $(G,T)$ with the root system of $(\widetilde{G},\widetilde{T})$. We denote by $\widetilde{\fg}$, resp.~$\widetilde{\ft}$, the Lie algebra of $\widetilde{G}$, resp.~$\widetilde{T}$. Note that in all cases the Weyl group of $(\widetilde{G},\widetilde{T})$ can be identified with $W$ and there exists a $W$-equivariant isomorphism $\fh \cong \Hom_{\E}(\widetilde{\ft}, \E)$. 

\begin{lem}
\label{lem:faithful-module}
There exists a free $\E$-module $V$, a representation $\sigma \colon \widetilde{G} \to \mathrm{GL}(V)$ whose differential $d\sigma$ is injective, and a subspace $\mathfrak{i} \subset \mathfrak{gl}(V)$ such that
\begin{enumerate}
\item $\mathfrak{gl}(V) = d\sigma(\widetilde{\fg}) \oplus \mathfrak{i}$;
\item $\mathfrak{i}$ is stable under the (adjoint) action of $\widetilde{G}$.
\end{enumerate}
\end{lem}

\begin{proof}
If $G=\mathrm{GL}_n(\Fpb)$ then $\widetilde{G} \cong \mathrm{GL}_n(\E)$ and we can take $V=\E^n$, $\mathfrak{i}=\{0\}$. So we assume that $\widetilde{G}$ is simply connected and quasi-simple, not of type $\mathbf{A}$. If $\E=\K$ or $\F$ the result is proved in \cite[Lemma I.5.3]{ss}. If $\E=\O$, one can repeat the arguments in the proof in \emph{loc.}~\emph{cit.} to choose a free $\O$-module $V$ and a representation $\sigma \colon \widetilde{G} \to \mathrm{GL}(V)$ whose differential is injective, and such that the symmetric bilinear form $\kappa \colon (X,Y) \mapsto \mathrm{tr}(XY)$ on $\mathfrak{gl}(V)$ restricts to a nondegenerate form on $d\sigma(\widetilde{\fg})$. (In fact it is enough to check the latter condition after specializing $\O$ to $\F$; in this case it is proved in \emph{loc.}~\emph{cit.}) Then the orthogonal complement of $\widetilde{\fg}$ with respect to $\kappa$ provides the desired choice for $\mathfrak{i}$, by Lemma \ref{lem:orthogonal}.
\end{proof}

\begin{proof}[Proof of Proposition {\rm \ref{prop:logarithm}}]
We fix $V$, $\sigma$, $\mathfrak{i}$ as in Lemma \ref{lem:faithful-module}.
We observe that the composition
\[
\widetilde{G} \xrightarrow{\sigma} \mathrm{GL}(V) \to \mathfrak{gl}(V) \to d\sigma(\widetilde{\fg}) \simto \widetilde{\fg}
\]
(where the second morphism sends $x$ to $x-1$ and the third one is the projection with respect to the decomposition $\mathfrak{gl}(V) = d\sigma(\widetilde{\fg}) \oplus \mathfrak{i}$) defines a $\widetilde{G}$-equivariant morphism sending $1$ to $0$ and whose differential at $1$ is the identity. Restricting to $\widetilde{T}$-fixed points, we obtain a $W$-equivariant morphism $\widetilde{T} \to \widetilde{\ft}$ with the same properties. The latter morphism is defined at the level of algebras of functions by a $W$-equivariant morphism $S \to \E\mathbf{Y}$ which sends $\fh S$ inside $\mathfrak{m}$. (Recall that we have $\fh \cong \Hom_{\E}(\widetilde{\ft}, \E)$ and $\mathbf{Y} \cong X^*(\widetilde{T})$.) Completing this map we obtain the desired isomorphism $\hS \simto \widehat{\E\bfY}$.
\end{proof}

From now on we fix an isomorphism as in Proposition~\ref{prop:logarithm}. More precisely we fix such an isomorphism for $\E=\O$, and deduce similar isomorphisms for $\K$ and $\F$ by extension of scalars. 

As explained in \cite[\S A.1]{by}, the construction of \cite[\S 5]{verdier} provides, for any $\cF$ in $\Db_{(B)}(\cB,\E)$, a group homomorphism (called \emph{monodromy})
\begin{equation}\label{eqn:monodromy-tate}
\mathbf{T}_\ell(T) \to \mathrm{Aut}(\cF),
\end{equation}
where $\mathbf{T}_\ell(T) = \varprojlim \{ z \in T(\Fpb) \mid z^{\ell^n} = 1 \}$ is the $\ell$-adic Tate module of $T$ and $\mathrm{Aut}(\cF)$ is the group of automorphisms of $\cF$.  More precisely, to define this morphism we consider the $U$-equivariant derived category $\Db_{U}(\cB,\E)$ where $U$ is the unipotent radical of $B$ (see Appendix~\ref{sec:equivariant-Db}), and we observe that the forgetful functor $\Db_{U}(\cB,\E) \to \Dbc(\cB,\E)$ is fully faithful (see Proposition~\ref{prop:For-Dequ}), with essential image $\Db_{(B)}(\cB,\E)$; hence we can identify these categories. Then we observe that $\Db_{U}(\cB,\E)$ is naturally equivalent to the category $\Db_{B^{\mathrm{op}}}(U \backslash G, \E)$, where the $B^{\mathrm{op}}$-action on $U \backslash G$ is induced by right multiplication on $G$. Now we have the $T$-torsor $U \backslash G \to B \backslash G$, so that we are in the setting of~\cite[\S A.1]{by}.

Let us fix once and for all a topological generator of the $\ell$-adic Tate module
\[
\mathbf{T}_\ell(\Gm) = \varprojlim \, \bigl\{ \zeta \in \Fpb^\times \mid \zeta^{\ell^n} = 1 \bigr\}.
\]
This choice determines a group homomorphism $\bfY \to \mathbf{T}_\ell(T)$, and hence, via~\eqref{eqn:monodromy-tate}, an action $\bfY \to \mathrm{Aut}(\cF)$ for any $\cF \in \Db_{(B)}(\cB,\E)$.  We extend this to an algebra map
\[
\mu'_{\cF} \colon \E \bfY \to \End(\cF).
\]

It is shown in \cite{verdier} that monodromy commutes with all morphisms in $\Db_{(B)}(\cB,\E)$, in the sense that if $x \in \E \bfY$, $\cF,\cG \in \Db_{(B)}(\cB,\E)$ and $f \colon \cF \to \cG$ is a morphism then we have $f \circ \mu'_{\cF}(x) = \mu'_{\cG}(x) \circ f$. (In other words, $\mu'_{(-)}$ defines a morphism from $\E\bfY$ to the algebra of endomorphisms of the identity functor of $\Db_{(B)}(\cB,\E)$.)

It is easily checked that for all $y \in \bfY$ and $\cF \in \Db_{(B)}(\cB,\E)$ the morphism $\mu'_{\cF}(y)$ is unipotent. Hence $\mu'_{\cF}$ can be extended to a morphism $\widehat{\E \bfY} \to \End(\cF)$. Composing with the isomorphism of Proposition~\ref{prop:logarithm} we obtain an algebra morphism
\[
\mu_{\cF} \colon \hS \to \End(\cF),
\]
which will play the role of the logarithm of $\mu'$ as considered in \cite{by}. This construction is compatible with extension of scalars in the natural sense.

\begin{rmk}
\label{rmk:mu-equivariant}
If $\cF \in \Perv_{(B)}(\cB,\E)$, then it follows from definitions that $\cF$ is $B$-equivariant iff $\mu'_\cF(\mathfrak{m})=\{0\}$, i.e.~iff $\mu_\cF(\fh)=\{0\}$.
\end{rmk}

\subsection{The functors $\mathbb{V}$ and $\mathbb{V}'$}
\label{ss:functor-V}

In this subsection we define the functor $\mathbb{V}$ which will provide the equivalence of Theorem \ref{thm:equivalence-tilting}, and a variant denoted $\mathbb{V}'$.

We begin with a lemma, whose proof can be copied from~\cite[Lemma~4.4.7]{by} or~\cite[\S 2.1]{bbm}.

\begin{lem}
\label{lem:delta-socle}
Assume that $\E=\K$ or $\F$.
For any $w \in W$,
\begin{enumerate}
\item the socle of $\Delta_w$ is isomorphic to $\IC_e$, and $[\Delta_w : \IC_e] = 1$;\label{it:delta-socle-delta}
\item the head of $\nabla_w$ is isomorphic to $\IC_e$, and $[\nabla_w : \IC_e] = 1$.\label{it:delta-socle-nabla}
\end{enumerate}
\end{lem}

In the following corollary we come back to the general setting where $\E$ is either $\K$, $\O$, or $\F$. Recall the object $\cP_e$ defined in~\S\ref{ss:reminder-radon}. 

\begin{cor}
\label{cor:Pe}
\begin{enumerate}
\item
For any $w \in W$ we have $(\cP_e : \Delta_w)=1$.
\item
There exist isomorphisms
$\F(\cP_e^\O) \cong \cP_e^\F$, $\K(\cP_e^\O) \cong \cP_e^\K$.
\end{enumerate}
\end{cor}

\begin{proof}
The first isomorphism in $(2)$ follows from the definition of $\cP_e^\O$; see~\S\ref{ss:decomposition}. Hence it is enough to prove $(1)$ when $\E=\K$ or $\F$. In this case, the claim follows from Lemma \ref{lem:delta-socle}$(2)$ and the reciprocity formula in the highest weight category $\Perv_{(B)}(\cB,\E)$. 
To prove the second isomorphism in $(2)$ we note that $\cP_e^\K$ is a direct summand in $\K(\cP_e^\O)$ (see again~\S\ref{ss:decomposition}). And by $(1)$ these objects have the same length,
so they must be isomorphic.
\end{proof}

\begin{rmk}
As suggested to us by G.~Williamson, one can also prove the second isomorphism in Corollary~\ref{cor:Pe}$(2)$ as follows. By the last equality in the proof of Theorem \ref{thm:decomposition}, the isomorphism is equivalent to the claim that $\IC_e^\F$ does not appear as a composition factor of any $\F(\IC_w^\O)$ with $w \neq e$. However if $w \neq e$ then $\IC^\O_w$ is a shifted pullback of a perverse sheaf on a minimal partial flag variety $\scP^s$ for some simple reflection $s$. Hence the same holds for $\F(\IC_w^\O)$; in particular, all the composition factors of this perverse sheaf are of the form $\IC_v^\F$ with $vs<v$ in the Bruhat order.
\end{rmk}

Let us set $A := \End(\cP_e)$. Observe that, by Lemma \ref{lem:Hom-proj-F}$(1)$, $A_\O$ is free as an $\O$-module. Moreover, if we fix isomorphisms as in Corollary \ref{cor:Pe}$(2)$, we obtain isomorphisms of algebras
\begin{equation}
\label{eqn:F-K-A}
\F(A_\O) \cong A_\F, \qquad \K(A_\O) \cong A_\K.
\end{equation}
The construction of~\S\ref{ss:monodromy} allows us to define a morphism
\begin{equation}
\label{eqn:monodromy-Pe}
S \hookrightarrow \hS \xrightarrow{\mu_{\cP_e}} \End(\cP_e) = A.
\end{equation}
This collection of morphisms is compatible with extension of scalars in the obvious sense.

\begin{lem}
\label{lem:monodromy-Pe}
The morphism \eqref{eqn:monodromy-Pe} factors through a morphism
$\phi \colon C \to A$.
If $\E=\K$, this morphism is an isomorphism. If $\E=\O$, it is injective and the cokernel is a torsion $\O$-module.
\end{lem}

\begin{proof}
These claims are well known in case $\E=\K$; see~\S\ref{ss:reminder-by} below. (See also \cite{soergel-kategorie} or \cite{bernstein} for earlier proofs based on representation theory.)
Now, consider the case $\E=\O$. By the remarks above
we have a commutative diagram
\[
\xymatrix@C=1.5cm@R=0.5cm{
S_\O \ar[d] \ar[r]^-{\eqref{eqn:monodromy-Pe}_{\O}} & A_\O \ar@{^{(}->}[d] \\
S_\K \ar[r]^-{\eqref{eqn:monodromy-Pe}_{\K}} & A_\K.
}
\]
Since the morphism on the second line is trivial on $(S_\K)^W_+$, we deduce that the morphism on the first line must be trivial on $(S_\O)^W_+$. The resulting morphism $\phi_\O$ is such that $\K(\phi_\O)=\phi_\K$ is an isomorphism (under the second isomorphism in Lemma \ref{lem:invariants}$(2)$). Since both $\O$-modules are free (see the remarks above and Lemma~\ref{lem:invariants}$(3)$), we deduce that $\phi_\O$ is injective, and that its cokernel is torsion.

Using a similar argument, the case $\E=\F$ follows from the case $\E=\O$ using Lemma \ref{lem:invariants}(1).
\end{proof}

Lemma~\ref{lem:monodromy-Pe} and~\eqref{eqn:F-K-A} imply in particular that the algebra
$A$ is commutative (since $C$ is). We will denote by $A \lh\underline{\mathsf{mod}}$, resp.~$C \lh\underline{\mathsf{mod}}$, the additive category of finitely generated $A$-modules, resp.~$C$-modules, which are free over $\E$.  Consider the functors
\[
\mathbb{V}' := \Hom(\cP_e,-) \colon \Tilt_{(B)}(\cB,\E) \to A \lh\underline{\mathsf{mod}}, \qquad \iota \colon A \lh\underline{\mathsf{mod}} \to C\lh\underline{\mathsf{mod}}
\]
(here $\iota$ is restriction of scalars with respect to $\phi$) 
and
\[
\mathbb{V}:=\iota \circ \mathbb{V}'.
\]
(In case $\E=\O$, the fact that $\mathbb{V}'$ takes values in $A\lh\underline{\mathsf{mod}}$ follows from Lemma \ref{lem:Hom-proj-F}$(1)$.) By Lemma \ref{lem:Hom-proj-F}$(1)$ again, the functor $\mathbb{V}'$ commutes with $\F(-)$ and $\K(-)$ via isomorphisms~\eqref{eqn:F-K-A}. The same is clearly true for $\iota$, and hence also for $\mathbb{V}$.

\begin{lem}
\label{lem:iota-V'}
\begin{enumerate}
\item
If $\E=\K$ or $\O$, the functor $\iota$ is fully faithful.
\item
If $\E=\K$ or $\F$, the functor $\mathbb{V}'$ is fully faithful.
\end{enumerate}
\end{lem}

\begin{proof}
(1) follows from Lemma \ref{lem:monodromy-Pe} (and the definition of $A \lh\underline{\mathsf{mod}}$ in case $\E=\O$).

(2) follows from Lemma \ref{lem:delta-socle}; see the arguments in~\cite[\S 2.1]{bbm}.
\end{proof}

\begin{rmk}
If $w=w_0$, by Remark \ref{rmk:Dw0} we have $D_{w_0}=C$. Moreover, once Theorem~\ref{thm:equivalence-tilting} is established, one can check that $\mathbb{V}(\cT_{w_0})=D_{w_0}$ (see~\S\ref{ss:nu} below), and that $\cT_{w_0} \cong \cP_e$ (see Proposition \ref{prop:Pe-Tw0} below). It follows that we have $A \cong C$ in all cases, and that $\iota$ is an equivalence. However we were not able to prove these facts directly, so that we have to distinguish these two algebras for the moment.
\end{rmk}

\subsection{Artin--Schreier sheaf and averaging functors}
\label{ss:AS-sheaf}

Recall that we have chosen $p$ such that there exists a primitive $p$-th root of unity in $\F$, and hence also in $\O$. Let us fix such a $p$-th root of unity, and denote by $\psi_\O \colon \Z/p\Z \hookrightarrow \O^\times$ the associated group homomorphism.  We define $\psi_\F$ and $\psi_\K$ in the obvious way. These morphisms allow us to define the Artin--Schreier sheaf $\AS_\psi^\E$ over $\Ga$ (where $\Ga$ is the additive group over $\Fpb$) for $\E=\O$, $\K$, or $\F$. (Recall that $\AS_\psi^\E$ is defined as the $\psi$-invariants in the direct image of the constant sheaf under the morphism $\Ga \to \Ga$ sending $x$ to $x^p-x$, a Galois cover with Galois group $\Z/p\Z$.)
By definition, this sheaf satisfies
\begin{equation}
\label{eqn:cohom-AS}
\mathbb{H}^\bullet(\Ga, \AS_\psi) = \mathbb{H}^\bullet_c(\Ga, \AS_\psi)=0,
\end{equation}
and it is multiplicative in the sense of Appendix \ref{sec:equivariant-Db}.

We denote by $\fg$ the Lie algebra of $G$ and, for $\alpha$ a root, by $\fg_\alpha \subset \fg$ the corresponding root subspace.
If $s$ is a simple reflection, associated with the simple root $\alpha$, we denote by $U_s \subset U$ the subgroup whose Lie algebra is $\fg_\alpha$, by $U_s^- \subset U^-$ the subgroup whose Lie algebra is $\fg_{-\alpha}$, and by $U^s \subset U$ the normal subgroup which is the unipotent radical of the minimal standard parabolic subgroup of $G$ associated with $s$. 

Let us fix, for each $s$, a group isomorphism $U_s^- \cong \Ga$, and denote by $\chi \colon U^- \to \Ga$ the morphism obtained as the composition
\[
U^- \twoheadrightarrow U^-/[U^-,U^-] \cong \prod_s U_s^- \cong \prod_s \Ga \xrightarrow{+} \Ga.
\]
(Here $s$ runs over simple reflections, ordered in an arbitrary way.) We set $\cL_\psi:=\chi^* \AS_\psi$. Then $\cL_\psi$ is a multiplicative local system in the sense of Appendix \ref{sec:equivariant-Db}, so that we can consider the equivariant derived category $\Db_{U^-, \cL_\psi}(\cB,\E)$. One can also consider the equivariant derived category $\Db_U(\cB,\E)$, which is canonically equivalent to $\Db_{(B)}(\cB,\E)$, as explained in~\S\ref{ss:monodromy}.

As in \cite{by}, for $?=\mathord{!}$ or $*$ we consider the functor
\[
\Av_{\psi,?} \colon \Db_{U}(\cB,\E) \xrightarrow{\For} \Dbc(\cB,\E) \xrightarrow{\av_?} \Db_{U^-, \cL_\psi}(\cB,\E).
\]
(See Appendix~\ref{sec:equivariant-Db} for the notation.)
Using \eqref{eqn:cohom-AS} and copying the arguments in \cite{bbrm} (see also \cite[Lemma 4.4.3]{by}), one can check that
the natural morphism of functors $\Av_{\psi,!} \to \Av_{\psi,*}$ is an isomorphism.
For simplicity, we will denote both of these functors by $\Av_\psi$. Similarly, for $?=\mathord{!}$ or $*$ we have a functor
\[
\Av_? \colon \Db_{U^-, \cL_\psi}(\cB,\E) \xrightarrow{\For} \Dbc(\cB,\E) \xrightarrow{\av_?} \Db_{U}(\cB,\E).
\]

The following lemma follows from Lemma~\ref{lem:av-adjunction}; the proof is identical to that of \cite[Lemma 4.4.5]{by}.

\begin{lem}
\label{lem:av-psi-adjunction}
The functor $\Av_\psi$ is right adjoint to $\Av_!$ and left adjoint to $\Av_*$.
\end{lem}

\subsection{Partial analogues}
\label{ss:partial}

For $s$ a simple reflection, we set $V_s:= U_s^- U^s$, and we denote by
$\cL_\psi^s$ the pullback of $\AS_\psi$ along the morphism
\[
V_s = U_s^- U^s \twoheadrightarrow U_s^- \simto \Ga.
\]
(Here, the first morphism is projection to $U_s^-$.) The local system $\cL_\psi^s$ is multiplicative, so as above we can consider the triangulated category $\Db_{V_s, \cL_\psi^s}(\cB,\E)$.

As observed in~\S\ref{ss:averaging-functors} we have a (partial) forgetful functor
$\For' \colon \Db_U(\cB,\E) \to \Db_{U^s}(\cB,\E)$.
On the other hand, for $?=!$ or $*$, since $U_s^-$ commutes with $U^s$, the functor $\mathrm{av}^{U_s^-, \cL^s_\psi}_?$ induces a functor $\Db_{U^s}(\cB,\E) \to \Db_{V_s, \cL_\psi^s}(\cB,\E)$.
(Here, by abuse we denote by $\cL_\psi^s$ the Artin--Schreier sheaf on $U_s^- \cong \Ga$.) We denote by
\[
\Av^s_{\psi,?} \colon \Db_{U}(\cB,\E) \to \Db_{V_s, \cL_\psi^s}(\cB,\E)
\]
the composition of these functors. As above one can show that the natural morphism $\Av^s_{\psi,!} \to \Av^s_{\psi,*}$ is an isomorphism, so that one can identify these functors and denote them by $\Av^s_\psi$.

Similarly, we have functors
\[
\Av^s_!, \ \Av^s_* \colon \Db_{V_s, \cL_\psi^s}(\cB,\E) \to \Db_{U}(\cB,\E)
\]
defined as the composition of the partial forgetful functor $\For'' \colon \Db_{V_s, \cL_\psi^s}(\cB,\E) \to \Db_{U^s}(\cB,\E)$ with the averaging functors $\Db_{U^s}(\cB,\E) \to \Db_{U}(\cB,\E)$ with respect to $U_s$. (To show that the functor $\mathrm{av}^{U_s}_?$ induces such a functor, we observe that any object in the essential image of $\Db_{U^s}(\cB,\E)$ in $\Dbc(\cB,\E)$ is of the form $\For^{U^s} \circ \mathrm{av}^{U^s}_?(\cG)$; then the claim follows from the observation that $\For^{U_s} \circ \mathrm{av}^{U_s}_? \circ \For^{U^s} \circ \mathrm{av}^{U^s}_? \cong \For^{U} \circ \mathrm{av}^{U}_?$.)

The proof of the following lemma is analogous to that of Lemma~\ref{lem:av-psi-adjunction}.

\begin{lem}
\label{lem:av-psi-adjunction-s}
The functor $\Av_\psi^s$ is right adjoint to $\Av_!^s$ and left adjoint to $\Av_*^s$.
\end{lem}

The same procedure also allows us to define a functor
\[
\widetilde{\Av}{}^s_! \colon \Db_{U^-, \cL_\psi}(\cB,\E) \xrightarrow{\For} \Db_{U_s^-, \cL^s_\psi}(\cB,\E) \xrightarrow{\av^{U^s}_!} \Db_{V_s, \cL_\psi^s}(\cB,\E).
\]

\begin{lem}
\label{lem:composition-av}
There exists an isomorphism of functors $\Av^s_! \circ \widetilde{\Av}{}^s_! \cong \Av_!$.
\end{lem}

\begin{proof}
The claim follows from the observation that the composition
\[
\Db_{U_s^-, \cL^s_\psi}(\cB,\E) \xrightarrow{\av^{U^s}_!} \Db_{V_s, \cL_\psi^s}(\cB,\E) \xrightarrow{\For} \Db_{U^s}(\cB,\E)
\]
is isomorphic to the composition
\[
\Db_{U_s^-, \cL^s_\psi}(\cB,\E) \xrightarrow{\For} \Dbc(\cB,\E) \xrightarrow{\av^{U^s}_!} \Db_{U^s}(\cB,\E)
\]
and easy facts on compositions of forgetful (resp.~averaging) functors.
\end{proof}

\subsection{Application to projective covers}
\label{ss:P_e^s}

From now on we will identify the categories $\Db_U(\cB,\E)$ and $\Db_{(B)}(\cB,\E)$ in the natural way, as explained in~\S\ref{ss:monodromy}.
With this identification, the various functors introduced in~\S\ref{ss:AS-sheaf} allow us (following~\cite{by}) to give an explicit construction of $\cP_e$.

\begin{lem}
\label{lem:average-Pe}
There exists an isomorphism $\Av_! \Av_\psi(\IC_e) \cong \cP_e$.
\end{lem}

\begin{proof}
By Lemma~\ref{lem:Hom-proj-F}(2) it is enough to prove the result in case $\E=\K$ or $\F$. The case $\E=\K$ is treated (in a slightly different setting) in~\cite[Lemma 4.4.11]{by}, and the case $\E=\F$ is similar: it is enough to prove that
\[
\Hom^i(\Av_! \Av_\psi(\IC_e), \IC_w) = \begin{cases}
\E & \text{if $w=e$ and $i=0$;} \\
0 & \text{otherwise}.
\end{cases}
\]
Using Lemma \ref{lem:av-psi-adjunction}, this easily follows from the fact that $\Av_\psi(\IC_w)=0$ for $w \neq e$; see the argument in~\cite[Lemma 4.4.6]{by}.
\end{proof}

\begin{rmk}
Under the isomorphism of Lemma~\ref{lem:average-Pe}, the projection $\cP_e \twoheadrightarrow \IC_e$ can be realized as the adjunction morphism $\Av_! \Av_\psi(\IC_e) \to \IC_e$.
\end{rmk}

To introduce further notation, we fix a simple reflection $s$. Below we will make use of the functor
\[
\xi_s := \Av^s_! \circ \Av^s_\psi \colon  \Db_{(B)}(\cB,\E) \to  \Db_{(B)}(\cB,\E).
\]
These functors over $\O$, $\K$ and $\F$ commute with $\K(-)$ and $\F(-)$. The same arguments as for Lemma~\ref{lem:average-Pe} allow to prove that we have
\begin{equation}
\label{eqn:xis-ICe}
\xi_s(\IC_e) \cong \cT_s.
\end{equation}
(Here we use the fact that $\cT_s$ is the projective cover of $\IC_e$ in the abelian category $\Perv_{(B)}(\overline{\cB_s},\E)$. This fact is well known; it can also be deduced from the case $G=\GL_2(\Fpb)$ of Proposition~\ref{prop:Pe-Tw0} below.) As above, if $\E=\K$ or $\F$ the adjunction morphism $\cT_s \cong \xi_s(\IC_e) \to \IC_e$ realizes $\IC_e$ as the head of $\cT_s$.

Finally, it is easy to check that $\xi_s$ commutes with convolution in the sense that, with the notation introduced in~\S\ref{ss:preliminary}, for any $\cF$ in $\Db_{(B)}(\cB,\E)$ and $\cG$ in $\Db_{B}(\cB,\E)$, there is a functorial isomorphism
\begin{equation}
\label{eqn:xi-convolution}
\xi_s(\cF \star^B \cG) \cong \xi_s(\cF) \star^B \cG.
\end{equation}

We set
\[
\cP_e^s:= \widetilde{\Av}{}^s_! \Av_\psi(\IC_e) \in \Db_{V_s,\cL_\psi^s}(\cB,\E).
\]
Using arguments similar to the ones above
one can check that, if $\E=\K$ or $\F$, $\cP_e^s$ is the projective cover in the category $\Perv_{V_s, \cL_\psi^s}(\cB,\E)$ of $(V_s, \cL_\psi^s)$-equivariant perverse sheaves of the simple object supported on $\overline{\cB_s}$, and also the indecomposable tilting object in this highest weight category whose support is maximal. (We will use these claims only in case $\E=\K$, in which case they follow from \cite[Corollary 5.5.2]{by}.) 

By construction and Lemma \ref{lem:composition-av} we have 
\begin{equation}
\label{eqn:Pes-Pe}
\Av^s_!(\cP_e^s) \cong \cP_e.
\end{equation}
In particular, there exists a natural morphism
\begin{equation}
\label{eqn:morphism-Pe-xis}
\cP_e \cong \Av^s_!(\cP_e^s) \to \Av^s_! \Av^s_\psi \Av^s_! (\cP_e^s) = \xi_s(\cP_e)
\end{equation}
where the middle arrow is induced by adjunction; see Lemma~\ref{lem:av-psi-adjunction-s}.

\subsection{The case $\E=\K$}
\label{ss:reminder-by}

In this subsection we assume that $\E=\K$, and we state some results which are known in this case, mainly thanks to \cite{by}.

\begin{lem}
\label{lem:End-Pes-K}
Assume $\E=\K$.
The morphism $\End(\cP_e^s) \to \End(\cP_e)$ induced by the functor $\Av_!^s$ (see \eqref{eqn:Pes-Pe}) is injective. Moreover, via the isomorphism of Lemma {\rm \ref{lem:monodromy-Pe}}, its image is identified with $C_s$. In other words, there exists a unique isomorphism $C_s \simto \End(\cP_e^s)$ which makes the following diagram commutative:
\[
\xymatrix@C=1.5cm@R=0.6cm{
C_s \ar[d]^-{\wr} \ar@{^{(}->}[r] & C \ar[d]_-{\wr}^-{{\rm Lem.}~\ref{lem:monodromy-Pe}} \\
\End(\cP_e^s) \ar[r]^-{\Av_!^s} & \End(\cP_e),
}
\]
\end{lem}

\begin{proof}
Thanks to the ``self-duality'' of \cite[\S 5.3]{by} and the ``paradromic-Whittavar\-i\-ant duality'' of \cite[\S 5.5]{by} (with $\Theta=\{s\}$) we have a commutative diagram
\[
\xymatrix@C=1.5cm@R=0.6cm{
\mathbb{H}^\bullet(\scPv^s;\K) \ar[d]^-{\wr} \ar@{^{(}->}[r]^-{({\check \pi}^s)^*} & \mathbb{H}^\bullet(\cBv;\K) \ar[d]^-{\wr} \\
\End(\cP_e^s) \ar[r]^-{\Av_!^s} & \End(\cP_e).
}
\]
Here we have used the following facts:
\begin{enumerate}
\item
under the equivalence of \cite[Theorem~5.3.1]{by}, (the mixed analogue of) $\cP_e=\cT_{w_0}$ corresponds to (the mixed analogue of) $\ICv_{w_0}$; 
\item
under the equivalence of \cite[Theorem~5.5.1]{by}, (the mixed analogue of) $\cP_e^s$---which coincides with the object denoted $\cT_{\overline{sw_0},\chi}$ in \cite{by}---corresponds to (the mixed analogue of) the constant perverse sheaf on $\scPv^s$;
\item
these equivalences intertwine (the mixed analogue of) the functors $\Av_!^s$ and $({\check \pi}^s)^*[1]$, cf.~\cite[Theorem~5.5.1$(2)$]{by};
\item
these equivalences also induce isomorphisms between $\Hom$-spaces in the ``non-mixed'' categories, cf.~\cite[Theorem~5.3.1(4) \& Theorem~5.5.1(5)]{by}.
\end{enumerate}
On the right-hand side one can identify $\mathbb{H}^\bullet(\cBv;\K)$ with $C$; see Proposition~\ref{prop:cohomology-G/B}. The resulting isomorphism $C \simto  \End(\cP_e)$ might not be the isomorphism of Lemma~\ref{lem:monodromy-Pe} (see Remark~\ref{rmk:monodromy-BY}), but the two isomorphisms differ by a $W$-equivariant automorphism of $C$. Then one concludes using Corollary~\ref{cor:cohomology-G/P}.
\end{proof}

Using Lemma~\ref{lem:End-Pes-K} we can consider the functors
\begin{align*}
\widetilde{\mathbb{V}}:=\Hom(\cP_e, -) \colon & \Perv_{U}(\cB,\K) \to C\lh\mathsf{mod}, \\
\widetilde{\mathbb{V}}^s:=\Hom(\cP_e^s, -) \colon & \Perv_{V_s, \cL_\psi^s}(\cB,\K) \to C^s\lh\mathsf{mod}.
\end{align*}
(Here, $\Perv_{U}(\cB,\K)$ and $\Perv_{V_s, \cL_\psi^s}(\cB,\K)$ are the categories of perverse sheaves in the categories $\Db_U(\cB,\K)$ and $\Db_{V_s,\cL^s_\psi}(\cB,\K)$, respectively; the t-structure is induced by the one on $\Dbc(\cB,\K)$.) Since $\cP_e$, resp.~$\cP_e^s$, is the projective cover of a simple object in $\Perv_{U}(\cB,\K)$, resp.~$\Perv_{V_s, \cL_\psi^s}(\cB,\K)$, these functors are
quotient functors, as explained in~\cite[\S 2.2]{soergel}.

By \eqref{eqn:Pes-Pe} and Lemma \ref{lem:av-psi-adjunction-s}, we have $\widetilde{\mathbb{V}}^s \circ \Av^s_\psi \cong \mathsf{Res}^C_{C_s} \circ \widetilde{\mathbb{V}}$, where $\mathsf{Res}^C_{C_s} \colon C\lh\mathsf{mod} \to C_s\lh\mathsf{mod}$ is the restriction functor. By a standard argument (see e.g.~\cite[Remark~2.2.4]{soergel}) one deduces a canonical isomorphism of functors $\widetilde{\mathbb{V}} \circ \Av^s_!(-) \cong C \otimes_{C_s} \widetilde{\mathbb{V}}^s(-)$, and hence a canonical isomorphism of functors
\begin{equation}
\label{eqn:V-xi-K}
C \otimes_{C_s} \widetilde{\mathbb{V}}(-) \simto
\widetilde{\mathbb{V}} \circ \xi_s(-).
\end{equation}
By construction, this isomorphism sends $x \otimes f$ (where $x \in C$ and $f \in \widetilde{\mathbb{V}}(\cF)=\Hom(\cP_e,\cF)$) to the composition
\[
\cP_e \xrightarrow{\phi(x)} \cP_e \xrightarrow{\eqref{eqn:morphism-Pe-xis}} \xi_s(\cP_e) \xrightarrow{\xi_s(f)} \xi_s(\cF).
\]

\subsection{The functors $\xi_s$ and tilting perverse sheaves}

In this subsection we fix a simple reflection $s$.

\begin{lem}
\label{lem:xi-Tilt}
\begin{enumerate}
\item
For all $w \in W$, $\xi_s(\Delta_w)$ is perverse, and admits a standard filtration with associated graded $\Delta_w \oplus \Delta_{sw}$.
\item
For all $w \in W$, $\xi_s(\nabla_w)$ is perverse, and admits a costandard filtration with associated graded $\nabla_w \oplus \nabla_{sw}$.
\item
$\xi_s$ restricts to an endofunctor of $\Tilt_{(B)}(\cB,\E)$.
\end{enumerate}
\end{lem}

\begin{proof}
Let us prove (1). First we assume that $sw>w$. Then by \eqref{eqn:xis-ICe} and \eqref{eqn:xi-convolution} we have
\[
\xi_s(\Delta_w) = \xi_s(\IC_e \star^B \Delta_w) \cong \xi_s(\IC_e) \star^B \Delta_w \cong \cT_s \star^B \Delta_w.
\]
Now $\cT_s$ fits in an exact sequence $\Delta_s \hookrightarrow \cT_s \twoheadrightarrow \Delta_e$, and we conclude using Lemma~\ref{lem:convolution-Delta-nabla}.
If $sw<w$ then similarly we have
\[
\xi_s(\Delta_w) = \xi_s(\Delta_s \star^B \Delta_{sw}) \cong \xi_s(\Delta_s) \star^B \Delta_{sw}.
\]
Now we have an exact sequence $\IC_e \hookrightarrow \Delta_s \twoheadrightarrow \IC_s$, and as in the proof of Lemma~\ref{lem:average-Pe} we have $\xi_s(\IC_s)=0$. Hence
$\xi_s(\Delta_s) \cong \xi_s(\IC_e)$. Then one can conclude as in the first case.

The proof of (2) is similar to that of (1). Finally, (3) is an obvious consequence of (1) and (2).
\end{proof}

If $\cT$ is in $\Tilt_{(B)}(\cB,\E)$, we consider the map
\begin{equation}
\label{eqn:morphism-V-xis}
C \otimes_\E \mathbb{V}(\cT) \to \mathbb{V} \bigl( \xi_s(\cT) \bigr)
\end{equation}
sending $x \in C$ and $f \in \mathbb{V}(\cT)=\Hom(\cP_e,\cT)$ to the composition
\[
\cP_e \xrightarrow{\phi(x)} \cP_e \xrightarrow{\eqref{eqn:morphism-Pe-xis}} \xi_s (\cP_e) \xrightarrow{\xi_s(f)} \xi_s(\cT).
\]

\begin{prop}
\label{prop:V-xis}
The morphism \eqref{eqn:morphism-V-xis} factors through an isomorphism of functors
\[
C \otimes_{C_s} \mathbb{V}(-) \simto \mathbb{V}\circ \xi_s(-).
\]
\end{prop}

\begin{proof}
In case $\E=\K$, this result was proved in~\S\ref{ss:reminder-by}; see in particular~\eqref{eqn:V-xi-K}. As in the proof of Lemma \ref{lem:monodromy-Pe}, one can deduce (using Lemma~\ref{lem:invariants}$(4)$) that \eqref{eqn:morphism-V-xis} factors through a morphism of functors
\[
C \otimes_{C_s} \mathbb{V}(-) \to \mathbb{V}\circ \xi_s(-),
\]
first for $\E=\O$ and then for $\E=\F$.

Using again an extension-of-scalars argument, to prove that the induced morphism is an isomorphism for $\E=\O$ or $\F$ it is enough to treat the case $\E=\F$. So we assume $\E=\F$ from now on. In this case it follows from Lemma \ref{lem:xi-Tilt} that $\xi_s$ restricts to an exact endofunctor of $\Perv_{(B)}(\cB,\F)=\Perv_{U}(\cB,\F)$. One can also extend $\mathbb{V}$ to a functor $\widetilde{\mathbb{V}} \colon \Perv_{(B)}(\cB,\F) \to C\lh\Mod$, and morphism~\eqref{eqn:morphism-V-xis} to a morphism of functors 
\[
C \otimes_{C_s} \widetilde{\mathbb{V}}(-) \to \widetilde{\mathbb{V}}\circ \xi_s(-).
\]
(Indeed, since $\widetilde{\mathbb{V}}$ is exact it is enough to prove that the morphism analogous to~\eqref{eqn:morphism-V-xis}, where now $\cF$ is projective in $\Perv_{(B)}(\cB,\F)$, factors as claimed. This follows from the analogous statement when $\E=\O$, which itself follows from the case $\E=\K$.)
We will prove that the latter morphism is an isomorphism.
To prove this, by the five lemma it is enough to prove that our morphism is an isomorphism when applied to $\IC_w$ for any $w \in W$.

First consider the case $w \neq e$. Then $\widetilde{\mathbb{V}}(\IC_w)=0$, and $\widetilde{\mathbb{V}}(\xi_s(\IC_w))=0$. (Indeed, $\IC_w$ is the shifted pullback of a perverse sheaf on a minimal partial flag variety $\scP^t$ for some simple reflection $t$. It follows that $\xi_s(\IC_w)$ has the same property, and hence cannot have $\IC_e$ as a composition factor.) Our morphism is thus trivially an isomorphism in this case.

Now, consider the case $w=e$. Then we have $\widetilde{\mathbb{V}}(\IC_e) = \F$, and $\widetilde{\mathbb{V}}(\xi_s(\IC_e))=\widetilde{\mathbb{V}}(\cT_s)$ by \eqref{eqn:xis-ICe}. The vector spaces $C \otimes_{C_s} \widetilde{\mathbb{V}}(\IC_e)$ and $\widetilde{\mathbb{V}}(\cT_s)$ both have dimension $2$, so to prove that our morphism is an isomorphism it is enough to prove that it is surjective. The object $\cT_s$ is indecomposable and has $\IC_e$ as its head, so there exists
a surjection $g \colon \cP_e \twoheadrightarrow \cT_s$.
Composition with $g$ induces an isomorphism
\[
\End(\cT_s) \simto \Hom(\cP_e,\cT_s).
\]
(Indeed, both vector spaces have dimension $2$, and this map is injective because $g$ is surjective.) Using this and the fact that monodromy commutes with all morphisms (see~\S\ref{ss:monodromy}), we deduce that it is enough to prove that any endomorphism of $\cT_s$ can be written as a composition
\[
\cT_s \xrightarrow{\mu_{\cT_s}(x)} \cT_s \to \xi_s(\cT_s) \twoheadrightarrow \xi_s(\IC_e) \cong \cT_s
\]
for some $x \in \hS$. (Here the second morphism is defined using adjunction in a way similar to \eqref{eqn:morphism-Pe-xis}, and the third one is induced by the projection $\cT_s \twoheadrightarrow \IC_e$.) This claim would follow if we prove the following properties:
\begin{enumerate}
\item
the composition $\cT_s \to \xi_s(\cT_s) \to \xi_s(\IC_e) \cong \cT_s$ is an isomorphism.
\item
the morphism $\mu_{\cT_s} \colon \hS \to \End(\cT_s)$
is surjective. 
\end{enumerate}
Property (1) follows from general results on adjunction. To prove (2) we observe that the perverse sheaf
$\cT_s$ is not $B$-equivariant. (Indeed one can easily check that
\[
\Ext^1_{\Perv_B(\cB,\F)}(\Delta_e,\Delta_s) = \Hom_{\Db_B(\cB,\F)}(\Delta_e, \Delta_s[1]) = 0,
\]
which implies our claim.) Hence there exists $y \in \bfY$ such that $\mu'_{\cT_s}(y)$ is unipotent but not equal to $\id$, see Remark~\ref{rmk:mu-equivariant}. Then the image of $\mu_{\cT_s}$ contains both $\id$ and $\mu_{\cT_s}'(y)-1$, so it is the whole of $\End(\cT_s)$.
\end{proof}

\subsection{Proof of Theorem \ref{thm:equivalence-tilting}}
\label{ss:equivalence-V}

Recall the $C$-modules $\mathsf{BS}(\us)$ and the category $\mathcal{S}_\E$ introduced in~\S\ref{ss:BS-modules}.

If $\us=(s_1, \cdots, s_i)$ is a sequence of simple reflections we set
\[
\cT(\us):= \xi_{s_i} \cdots \xi_{s_1}(\IC_e).
\]
By Lemma \ref{lem:xi-Tilt}, this object is a tilting perverse sheaf, i.e.~an object of $\Tilt_{(B)}(\cB,\E)$.
Moreover, by Lemma \ref{prop:V-xis} we have
\begin{equation}
\label{eqn:V-BS}
\mathbb{V}(\cT(\us)) \cong \mathsf{BS}(\us).
\end{equation}
Since any indecomposable object in $\Tilt_{(B)}(\cB,\E)$ appears as a direct summand of $\cT(\us)$ for some $\us$ (as follows from Lemma~\ref{lem:xi-Tilt}),
we deduce the following proposition.

\begin{prop}
\label{prop:V-S}
The functor $\mathbb{V}$ factors through a functor $\Tilt_{(B)}(\cB,\E) \to \mathcal{S}_\E$ (which will be denoted similarly).
\end{prop}

Finally we are in a position to finish the proof of Theorem~\ref{thm:equivalence-tilting}.

\begin{proof}[Proof of Theorem~{\rm \ref{thm:equivalence-tilting}}]
The compatibility of $\mathbb{V}$ with $\F(-)$ and $\K(-)$ was noted in~\S\ref{ss:functor-V}.

Given \eqref{eqn:V-BS}, as in the proof of Theorem \ref{thm:equivalence-parity}, to prove the equivalence it is enough to prove that $\mathbb{V}$ is fully faithful. If $\E=\K$, this claim is known (see Lemma~\ref{lem:iota-V'}), so we only consider the case $\E=\O$ or $\F$. Let $\cT_1,\cT_2 \in \Tilt_{(B)}(\cB,\O)$, and set $\cT_i^\F:=\F(\cT_i)$. Consider the following commutative diagram:
\[
\xymatrix{
\F \Hom \bigl( \cT_1,\cT_2 \bigr) \ar[d]^-{\wr} \ar[r]^-{a} & \F \Hom_{A_\O} \bigl( \mathbb{V}'(\cT_1),\mathbb{V}'(\cT_2) \bigr) \ar[d]^-{b} \ar[r]^-{\sim} & \F \Hom_{C_\O} \bigl(\mathbb{V}(\cT_1), \mathbb{V}(\cT_2) \bigr) \ar[d]_-{\wr} \\
\Hom \bigl( \cT_1^\F,\cT_2^\F \bigr) \ar[r]^-{\sim} & \Hom_{A_\F} \bigl( \mathbb{V}'(\cT_1^\F),\mathbb{V}'(\cT_2^\F) \bigr) \ar[r]^-{c} &  \Hom_{C_\F} \bigl(\mathbb{V}(\cT_1^\F), \mathbb{V}(\cT_2^\F) \bigr).
}
\]
Here we have simplified $\F \otimes_\O \Hom(-,-)$ to $\F\Hom(-,-)$. On each line, the first morphism is induced by $\mathbb{V}'$, and the second one by $\iota$ (so that the morphism from left to right is induced by $\mathbb{V}$). The invertibility of the second morphism on the first line (resp.~of the first morphism on the second line) follows from Lemma \ref{lem:iota-V'}. The invertibility of the rightmost vertical morphism follows from Corollary \ref{cor:morphisms-BS}, in view of Proposition \ref{prop:V-S}. Finally, the invertibility of the leftmost vertical morphism follows from Proposition \ref{prop:properties-tilting}$(2)$.

Using the left part of the diagram we obtain that $b$ is surjective, and using the right part we obtain that $b$ is injective. Hence it is an isomorphism, and then $a$ and $c$ are also isomorphisms. This already proves fullness and faithfulness in case $\E=\F$. To prove it in case $\E=\O$, we simply remark that the morphism
\[
\Hom \bigl( \cT_1,\cT_2 \bigr) \to \Hom_{C_\O} \bigl(\mathbb{V}(\cT_1), \mathbb{V}(\cT_2) \bigr)
\]
induced by $\mathbb{V}$ is a morphism between free $\O$-modules of finite rank (see Proposition \ref{prop:properties-tilting}$(2)$ for the left-hand side) which becomes invertible after applying $\F \otimes_\O (-)$; hence it must be an isomorphism.
\end{proof}

\subsection{Complement: comparison of $\cP_e$ and $\cT_{w_0}$}
\label{ss:Pe-Tw0}

For completeness, we conclude this section with a (geometric) proof of the fact that $\cP_e \cong \cT_{w_0}$. This fact is not used in the paper (except in the case $G=\GL_2(\Fpb)$ in~\S\ref{ss:P_e^s}), and is well known when $\E=\K$.

\begin{lem}\label{lem:delta-map}
Assume that $\E=\K$ or $\F$.
For any $w \in W$, we have:
\begin{enumerate}
\item $\dim \Hom(\Delta_w, \Delta_{w_0})=1$, and every nonzero map $\Delta_w\to \Delta_{w_0}$ is injective;\label{it:delta-map-delta}
\item $\dim \Hom(\nabla_{w_0}, \nabla_w)=1$, and every nonzero map $\nabla_{w_0} \to \nabla_w$ is surjective.\label{it:delta-map-nabla}
\end{enumerate}
\end{lem}

\begin{proof}
We only prove~$(1)$; the proof of $(2)$ is similar. It is well known (see e.g.~\cite[\S 2.2]{bbm}) that the functor
\[
(-) \star^B \Delta_w \colon \Db_{(B)}(\cB,\E) \to \Db_{(B)}(\cB,\E)
\]
is an equivalence of categories, with inverse $(-) \star^B \nabla_{w^{-1}}$. Hence we have
\[
\Hom(\Delta_e, \Delta_{w_0 w^{-1}}) \cong \Hom(\Delta_e \star^B \Delta_w, \Delta_{w_0 w^{-1}} \star^B \Delta_w) \cong \Hom(\Delta_w, \Delta_{w_0})
\]
(see Lemma \ref{lem:convolution-Delta-nabla}).
We deduce the first claim using Lemma~\ref{lem:delta-socle}. 
\end{proof}

\begin{lem}\label{lem:proj-std-mult}
Assume that $\E=\K$ or $\F$.
For any $w \in W$, we have:
\[
(\cT_{w_0} : \Delta_w) = (\cT_{w_0} : \nabla_w) = 1.
\]
\end{lem}

\begin{proof}
The first equality follows from the fact that $\cT_{w_0}$ is self-dual under Verdier duality. The second equality follows from Corollary \ref{cor:Pe}(1) and \eqref{eqn:multiplicities-T-P}.
\end{proof}

\begin{prop}
\label{prop:Pe-Tw0}
We have $\cT_{w_0} \cong \cP_e$.
\end{prop}

\begin{proof}
By Lemma~\ref{lem:Hom-proj-F}$(2)$ it is enough to prove the claim when $\E=\K$ or $\F$, which we will assume in the proof.
Lemma \ref{lem:proj-std-mult} implies that
\begin{equation}\label{eqn:dim-tw0-nabla}
\dim \Hom(\cT_{w_0}, \nabla_w) = 1
\end{equation}
for all $w \in W$. 
We claim that 
\begin{equation}
\label{eqn:tw0-nabla-surj}
\text{any nonzero map $\cT_{w_0} \to \nabla_w$ is surjective.} 
\end{equation}
Indeed, in the special case $w = w_0$ this is clear.  Otherwise, it follows from~\eqref{eqn:dim-tw0-nabla} and Lemma~\ref{lem:delta-map}\eqref{it:delta-map-nabla}.

Next, we claim that
\begin{equation}\label{eqn:tw0-icw}
\Hom(\cT_{w_0}, \IC_w) = 0 \qquad\text{if $w \ne e$.}
\end{equation}
Indeed, if there were a nonzero map $\cT_{w_0} \to \IC_w$, we could compose it with the inclusion $\IC_w \to \nabla_w$ to get a nonzero, nonsurjective map $\cT_{w_0} \to \nabla_w$, contradicting~\eqref{eqn:tw0-nabla-surj}.

On the other hand, we know from~\eqref{eqn:dim-tw0-nabla} that $\dim \Hom(\cT_{w_0}, \IC_e) = 1$.  Combining this with~\eqref{eqn:tw0-icw}, we see that $\cT_{w_0}$ has a unique simple quotient, isomorphic to $\IC_e$.  It must therefore be isomorphic to a quotient of $\cP_e$. But Corollary~\ref{cor:Pe} and Lemma~\ref{lem:proj-std-mult} imply that $\cT_{w_0}$ and $\cP_e$ have the same length, so $\cT_{w_0} \cong \cP_e$.
\end{proof}

\section{Proof of Theorem \ref{thm:main}}
\label{sec:proof}

In this section we finish the proof of Theorem \ref{thm:main}.  We retain the conventions of~\S\ref{ss:new-notation}, and we assume from now on that $\#\F > 2$.  In particular, we continue to assume that $G$ is a product of groups isomorphic either to $\GL_n(\C)$ or to an adjoint simple group that is not of type $\mathbf{A}$, so that the results of Sections~\ref{sec:soergel-modules}--\ref{sec:tilting-V} are available.  As explained at the beginning of Section~\ref{sec:soergel-modules}, it suffices to prove Theorem~\ref{thm:main} for groups of this form.

\subsection{Construction of $\nu$}
\label{ss:nu}

We define the functor $\nu$ as the composition:
\[
\Parity_{(\Bv)}(\cBv,\E) \xrightarrow[\sim]{{\rm Thm}.~\ref{thm:equivalence-parity}} \cS^\gr \xrightarrow{\For} \cS \xrightarrow[\sim]{{\rm Thm}.~\ref{thm:equivalence-tilting}} \Tilt_{(B)}(\cB,\E),
\]
and the isomorphism $\varepsilon$ in the obvious way.
It is clear by construction that this functor satisfies condition $(1)$ of Theorem \ref{thm:main}. Condition $(4)$ follows from the similar properties of the functors $\mathbb{H}$ and $\mathbb{V}$ proved in Theorems \ref{thm:equivalence-parity} and \ref{thm:equivalence-tilting}, respectively.

Let us prove now that $\nu$ also satisfies condition $(2)$. In fact, by definition we have $\mathbb{H}(\cEv_w)=D_w^\gr$ and $D_w=\For(D_w^\gr)$. Hence it is enough to check that $\mathbb{V}(\cT_w) \cong D_{w^{-1}}$. However, using Lemma~\ref{lem:xi-Tilt} one can check that, if $w=s_1 \cdots s_{\ell(w)}$ is any reduced decomposition of $w$, then $\cT_w$ is characterized by the fact that it appears as a direct summand of $\cT(s_{\ell(w)}, \ldots, s_{1})$, but does not appear as a direct summand of any $\cT(\us)$ where $\us$ is a sequence of simple reflections of length strictly less than $\ell(w)$. Since $D_{w^{-1}}$ admits a similar characterization in terms of modules $\mathsf{BS}(\us)$ (see~\S\ref{ss:modular-BS}), we conclude using~\eqref{eqn:V-BS}.

\subsection{Proof of condition $(3)$}

We will deduce condition $(3)$ of Theorem \ref{thm:main} from $(1)$ and $(2)$. In fact, it is enough to prove the formula when $\cEv=\cEv_{w^{-1}}$ for some $w \in W$. Then by $(2)$ the formula reduces to
\begin{equation}
\label{eqn:multiplicity-T-rk}
(\cT_w : \nabla_v) = \mathrm{rk}_\E \bigl( \mathbb{H}^\bullet(\cBv_{v^{-1}}, \iv_{v^{-1}}^* \cEv_{w^{-1}}) \bigr).
\end{equation}

We first note that for $v,w \in W$ we have
\[
\mathrm{rk}_\E \Hom(\cT_v, \cT_w) = \sum_{u \in W \atop u \leq v, \, u \leq w} (\cT_v : \Delta_u) \cdot (\cT_w : \nabla_u) = \sum_{u \in W \atop u \leq v, \, u \leq w} (\cT_v : \nabla_u) \cdot (\cT_w : \nabla_u).
\]
(In the second equality we have used the Verdier self-duality of $\cT_v$.)
This formula allows us to compute the numbers $(\cT_w : \nabla_u)$ by induction if one knows the ranks of the $\Hom$-spaces, using the formula
\begin{equation}
\label{eqn:multiplicities-tilting}
(\cT_w : \nabla_u) = \mathrm{rk}_\E \Hom(\cT_u, \cT_w) - \sum_{x<u} (\cT_u : \nabla_x) \cdot (\cT_w : \nabla_x)
\end{equation}
for $u \leq w$.

We have a similar formula for parity sheaves. In fact from \cite[Proposition 2.6]{jmw} we deduce
\begin{multline*}
\mathrm{rk}_\E \Hom^\bullet(\cEv_{v^{-1}}, \cEv_{w^{-1}}) = \sum_{u \in W \atop u \leq v, \, u \leq w} \mathrm{rk}_\E \mathbb{H}^\bullet(\cBv_{u^{-1}}, \iv_{u^{-1}}^* \cEv_{v^{-1}}) \cdot \mathrm{rk}_\E \mathbb{H}^\bullet(\cBv_{u^{-1}}, \iv_{u^{-1}}^! \cEv_{w^{-1}}) \\
= \sum_{u \in W \atop u \leq v, \, u \leq w} \mathrm{rk}_\E \mathbb{H}^\bullet(\cBv_{u^{-1}}, \iv_{u^{-1}}^* \cEv_{v^{-1}}) \cdot \mathrm{rk}_\E \mathbb{H}^\bullet(\cBv_{u^{-1}}, \iv_{u^{-1}}^* \cEv_{w^{-1}}).
\end{multline*}
(Here again, the second equality uses Verdier self-duality of $\cEv_{w^{-1}}$. We also use the fact that $\cEv_{v^{-1}}$, resp.~$\cEv_{w^{-1}}$, is supported on the closure of $\cBv_{v^{-1}}$, resp.~$\cBv_{w^{-1}}$.) This formula gives rise to an inductive way of computing the ranks of the cohomology groups of stalks, using the formula
\begin{multline}
\label{eqn:rank-fibers}
\mathrm{rk}_\E \mathbb{H}^\bullet(\cBv_{u^{-1}}, \iv_{u^{-1}}^* \cEv_{w^{-1}}) = \mathrm{rk}_\E \Hom^\bullet(\cEv_{u^{-1}}, \cEv_{w^{-1}}) \\
- \sum_{x<u} \mathrm{rk}_\E \mathbb{H}^\bullet(\cBv_{x^{-1}}, \iv_{x^{-1}}^* \cEv_{u^{-1}}) \cdot \mathrm{rk}_\E \mathbb{H}^\bullet(\cBv_{x^{-1}}, \iv_{x^{-1}}^* \cEv_{w^{-1}}).
\end{multline}
if $u \leq w$.

Comparing formulas \eqref{eqn:multiplicities-tilting} and \eqref{eqn:rank-fibers} and using $(1)$, one easily proves \eqref{eqn:multiplicity-T-rk} by induction on $v$.

\appendix

\section{(Twisted) equivariant derived categories\\ for acyclic groups}
\label{sec:equivariant-Db}

\subsection{Definitions}

In this section $X$ is either a complex algebraic variety equipped with the classical topology (in which case $\bk$ is an arbitrary noetherian commutative ring of finite global dimension) or a variety over $\Fpb$ equipped with the {\'e}tale topology (in which case $\bk$ is a finite extension of $\Ql$, or its ring of integers, or a finite field of characteristic $\ell$, with $\ell \neq p$). We denote by $\Dbc(X,\bk)$ the constructible derived category of $\bk$-sheaves on $X$. We assume that an algebraic group $V$ (either over $\C$ or over $\Fpb$) which satisfies
\begin{equation}
\label{eqn:V-acyclic}
\mathbb{H}^\bullet(V;\bk)=\bk
\end{equation}
acts on $X$, with action morphism $a \colon V \times X \to X$. (This condition implies in particular that $V$ is connected.)

We let $\cX$ be a rank-one local system on $V$ which is a \emph{multiplicative sheaf}, i.e.~which is endowed with an isomorphism
\begin{equation}
\label{eqn:X-multiplicative}
m^*\cX \simto \cX \boxtimes \cX
\end{equation}
(where $m \colon V \times V \to V$ is the multiplication map on $V$) satisfying the obvious associativity condition.

\begin{defn}
A \emph{$(V,\cX)$-equivariant complex} on $X$ is a pair $(\cF,\beta)$ where $\cF \in \Dbc(X,\bk)$ and $\beta$ is an isomorphism $a^*\cF \simto \cX \boxtimes \cF$ satisfying the usual cocycle condition.  A \emph{morphism} of $(V,\cX)$-equivariant complexes is a morphism in $\Dbc(X,\bk)$ that ``commutes with $\beta$'' in the obvious sense.  The category of $(V,\cX)$-equivariant complexes on $X$ is denoted $\Db_{V,\cX}(X,\bk)$.
\end{defn}

In the case where $\cX=\underline{\bk}_V$, we abbreviate $\Db_{V,\cX}(X,\bk)$ to $\Db_{V}(X,\bk)$. (We will see in Remark \ref{rmk:equiv-BL} below that the category $\Db_{V}(X,\bk)$ is equivalent to the $V$-equivariant derived category of $X$ in the sense of Bernstein--Lunts~\cite{bl}.) 

We let
\[
\For \colon \Db_{V,\cX}(X,\bk) \to \Dbc(X,\bk)
\]
be the forgetful functor.  

\begin{rmk}
\label{rmk:Dbequ-trivial}
Consider the case $X=V \times Y$, where $Y$ is any variety and $V$ acts by left multiplication on itself. Let $p_2 \colon X \to Y$ be the projection. Then $p_2^*$ defines in an obvious way a functor from $\Dbc(Y,\bk)$ to $\Db_V(X,\bk)$. We claim that this functor is an equivalence of categories. Indeed, by \eqref{eqn:V-acyclic} the functor $p_2^*$ induces an isomorphism
\[
\Hom_{\Dbc(Y,\bk)} (\cF,\cG) \simto \Hom_{\Dbc(X,\bk)}(p_2^* \cF, p_2^* \cG).
\]
From this it easily follows that the morphism
\[
 \Hom_{\Db_V(X,\bk)}(p_2^* \cF, p_2^* \cG) \to  \Hom_{\Dbc(X,\bk)}(p_2^* \cF, p_2^* \cG)
\]
induced by $\For$ is an isomorphism, and then that $p_2^*$ is fully-faithful. To show that it is essentially surjective it suffices to restrict the isomorphism $\beta \colon (m \times \id_Y)^* \cF \simto \underline{\bk}_V \boxtimes \cF$ to the image of the embedding $X \to V \times X$ given by $(v,y) \mapsto (v,1,y)$.

The functor $\Db_V(X,\bk) \to \Dbc(Y,\bk)$ defined by $p_{2*}$ provides a quasi-inverse to $p_2^*$.
\end{rmk}

\subsection{Averaging functors}
\label{ss:averaging-functors}

We will also consider two ``averaging'' functors
\[
\begin{aligned}
\av_! & \colon \Dbc(X,\bk) \to \Db_{V,\cX}(X,\bk), \\
\av_* & \colon \Dbc(X,\bk) \to \Db_{V,\cX}(X,\bk)
\end{aligned}
\qquad\text{by}\qquad
\begin{aligned}
\av_!(\cF) &= a_!(\cX \boxtimes \cF)[\dim V], \\
\av_*(\cF) &= a_*(\cX \boxtimes \cF)[\dim V].
\end{aligned}
\]
In both cases, the isomorphism $\beta$ is the natural one, obtained from the K{\"u}nneth formula and the base change theorem applied to the cartesian square
\[
\xymatrix@C=1.5cm@R=0.5cm{
V \times V \times X \ar[r]^-{m \times \id_X} \ar[d]_-{\id_V \times a} & V \times X \ar[d]^-{a} \\
V \times X \ar[r]^-{a} & X \\
}
\]
using condition \eqref{eqn:X-multiplicative} and, in the second case, the fact that $a$ and $m$ are smooth morphisms.

\begin{lem}
\label{lem:av-adjunction}
The functor $\av_!$ is left-adjoint to $\For[\dim V]$.  Similarly, $\av_*$ is right-adjoint to $\For[-\dim V]$.
\end{lem}

\begin{proof}
Let $p_1 \colon V \times X \to V$ and $p_2 \colon V \times X \to X$ be the projection maps. By Remark \ref{rmk:Dbequ-trivial} the functor $p_2^*$ defines an equivalence from $\Dbc(X,\bk)$ to $\Db_V(V \times X, \bk)$ (where $V$ acts on $V \times X$ via left multiplication on $V$), and the functor $p_{2*} \colon \Db_V(V \times X, \bk) \to \Dbc(X,\bk)$ is a quasi-inverse (in particular a right adjoint) to $p_2^*$. Similar remarks apply to $p_2^!$ and $p_{2!}$.

Let $Q \colon \Db_V(V \times X) \to \Db_{V,\cX}(V \times X)$ be the functor given by $Q(\cG) = p_1^*\cX \otimes \cG$.  This functor is an equivalence of categories; the inverse is given by $\cG \mapsto p_1^*\cX^{-1} \otimes \cG$, where $\cX^{-1}$ is the dual local system to $\cX$ on $V$.  Now, $\av_! \cong a_![\dim V] \circ Q \circ p_2^*$, where $a_!$ is considered as a functor from $\Db_{V,\cX}(V \times X)$ to $\Db_{V,\cX}(X,\bk)$. So $\av_!$ is left-adjoint to $p_{2*} \circ Q^{-1} \circ a^![-\dim V] \cong p_{2*} \circ Q^{-1} \circ a^*[\dim V]$.  The latter is clearly isomorphic to $\For[\dim V]$.

Similarly, $\av_* \cong a_* \circ Q \circ p_2^*[\dim V] \cong a_* \circ Q \circ p_2^![-\dim V]$ is right-adjoint to $p_{2!} \circ Q^{-1} \circ a^*[\dim V] \cong \For[-\dim V]$.
\end{proof}

\begin{lem}
\label{lem:av-For}
The composition $\av_! \circ \For[\dim V] \colon \Db_{V,\cX}(X,\bk) \to \Db_{V,\cX}(X,\bk)$ is isomorphic to the identity functor.
\end{lem}

\begin{proof}
By Lemma \ref{lem:av-adjunction} we have an adjunction morphism
$\av_! \circ \For[\dim V] \to \id$, and it suffices to prove that this morphism is an isomorphism. Now if $\cF$ is in $\Db_{V,\cX}(X,\bk)$ we have
\begin{multline*}
\av_! \circ \For(\cF)[\dim V] = a_!(\cX \boxtimes \cF)[2 \dim V] \cong a_! a^*(\cF) [2\dim V] \\
\cong \cF \lotimes_\bk a_!(\underline{\bk}_{V \times X}) [2\dim V]
\end{multline*}
by the projection formula. However by \eqref{eqn:V-acyclic} we have $\mathbb{H}^\bullet_c(V;\bk) = \bk[-2\dim V]$, hence $a_!(\underline{\bk}_{V \times X}) \cong \underline{\bk}_X[-2\dim V]$, and the claim follows.
\end{proof}

As an immediate consequence, we obtain the following facts.

\begin{prop}
\label{prop:For-Dequ}
\begin{enumerate}
\item
The functor $\For \colon \Db_{V,\cX}(X,\bk) \to \Dbc(X,\bk)$ is fully faithful. An object $\cF \in \Dbc(X,\bk)$ is in the essential image of $\Db_{V,\cX}(X,\bk)$ if and only if the adjunction morphism $\cF \to \For(\av_!\cF[\dim V])$ is an isomorphism.
\item
The essential image of $\For \colon \Db_{V,\cX}(X,\bk) \to \Dbc(X,\bk)$ is a triangulated subcategory of $\Dbc(X,\bk)$. In particular, the category $\Db_{V,\cX}(X,\bk)$ has a natural triangulated structure.
\end{enumerate}
\end{prop}

\begin{proof}
$(1)$ is a corollary of Lemma \ref{lem:av-For}. Then $(2)$ follows from the description of the essential image of $\For$ together with standard facts on triangulated categories.
\end{proof}

\begin{rmk}
\label{rmk:equiv-BL}
In the case $\cX=\underline{\bk}_V$, the category $\Db_V(X,\bk)$ is equivalent to the constructible $V$-equivariant derived category in the sense of Bernstein--Lunts. Indeed by \cite[Theorem 3.7.3]{bl}, under our assumption \eqref{eqn:V-acyclic} the latter is also equivalent to a full triangulated subcategory of $\Dbc(X,\bk)$. Moreover this full subcategory is generated (as a triangulated subcategory) by constructible $V$-equivariant sheaves in the sense of \cite[\S 0.2]{bl}, see \cite[Proposition 2.5.3]{bl}. The similar claim for the essential image of our functor $\For \colon \Db_{V}(X,\bk) \to \Dbc(X,\bk)$ is easy to check, which proves the claim.
\end{rmk}

Finally we remark that if $V' \subset V$ is a closed subgroup which also satisfies~\eqref{eqn:V-acyclic}, the restriction $\cX'$ of $\cX$ to $V'$ is multiplicative, and the functor $\For$ factors through a (fully faithful) functor $\For' \colon \Db_{V,\cX}(X,\bk) \to \Db_{V',\cX'}(X,\bk)$.

\section{Tilting perverse $\O$-sheaves and Radon transform \\ (joint with Geordie Williamson\protect\footnote{Max-Planck-Institut f{\"u}r Mathematik, Vivatsgasse~7, 53111, Bonn, Germany. E-mail: \texttt{geordie@mpim-bonn.mpg.de}.})}
\label{sec:tilting}

In this section we denote by $\O$ the ring of integers in a finite extension $\K$ of $\Ql$, and by $\F$ its residue field. We use the letter $\E$ to denote either of $\K$, $\O$ or $\F$.

\subsection{Notation}

Let us consider an algebraic variety $X$ endowed with a finite stratification
\[
X=\bigsqcup_{s \in \scS} X_s
\]
by locally closed subvarieties, and denote by $i_s \colon X_s \to X$ the inclusion.
We assume that each $X_s$ is isomorphic to an affine space, and consider the derived category $\Db_\scS(X,\O)$ of complexes whose cohomology is constant on each stratum $X_s$, and the subcategory of perverse sheaves $\Perv_{\scS}(X,\O)$. We assume either that $X$ is defined over $\C$ and work with the classical topology as in the main body of the paper, or that $X$ is defined over $\Fpb$, with $p \neq \ell$, and work with {\'e}tale sheaves as in \cite{rsw}. In the latter case we assume that the following condition is satisfied:
\begin{equation}
\label{eqn:strata-assumption}
\text{for all $s,t \in \scS$ and $n \in \Z$, $\mathcal{H}^n(i_t^*i_{s*}\underline{\O}_{X_s})$ is constant.}
\end{equation}

For any $s \in \scS$ we denote by $i_s \colon X_s \to X$ the inclusion, and consider the objects
\[
\Delta_s := i_{s!} \underline{\O}_{X_s} [\dim X_s], \qquad \nabla_s:= i_{s*} \underline{\O}_{X_s} [\dim X_s].
\]
Recall that, in this generality, a perverse sheaf $\cT$ in $\Perv_{\scS}(X,\O)$ is called \emph{tilting} if it admits both a standard filtration (i.e.~a filtration with subquotients of the form $\Delta_t$, $t \in \scS$) and a costandard filtration (i.e.~a filtration with subquotients of the form $\nabla_t$, $t \in \scS$).

We can similarly consider the categories $\Db_\scS(X,\F)$ and $\Perv_{\scS}(X,\F)$ of sheaves with coefficients in $\F$, the objects
\[
\Delta_s^\F := i_{s!} \underline{\F}_{X_s} [\dim X_s], \qquad \nabla_s^\F:= i_{s*} \underline{\F}_{X_s} [\dim X_s],
\]
and the corresponding notions of standard or costandard filtrations, and tilting perverse sheaves.

We have a ``modular reduction'' functor
\[
\F := \F \lotimes_\O (-) \colon \Db_\scS(X,\O) \to \Db_\scS(X,\F)
\]
which commutes with all the usual sheaf operations (derived direct and inverse images with or without support, derived tensor product, bifunctors $R\Hom$ and $R\mathcal{H} \hspace{-0.5pt} om$).
In particular, this implies that
\begin{equation}
\label{eqn:F-Delta-nabla}
\F(\Delta_s) \cong \Delta_s^\F \qquad \text{and} \qquad \F(\nabla_s) \cong \nabla_s^\F.
\end{equation}

\subsection{Tilting perverse sheaves: existence}

In this subsection we show how standard arguments giving constructions of
tilting modules in highest weight categories (see \cite{ringel, donkin, mathieu}) can
be adapted to prove the existence of ``enough'' tilting objects in
$\Perv_{\scS}(X,\O)$.

First, the following easy result can be proved as in \cite[Proposition 1.3]{bbm}.

\begin{lem}
\label{lem:nabla-flag}
Let $\cF$ be in $\Perv_{\scS}(X,\O)$. Then $\cF$ admits a costandard filtration iff for any $t \in \scS$, the object $i_t^! \cF$ is a direct sum of copies of $\underline{\O}{}_{X_t} [\dim X_t]$.
\end{lem}

\begin{prop}
\label{prop:existence-tiltings}

For any $s \in \scS$ there exists a tilting object $\cT$ in $\Perv_{\scS}(X,\O)$ supported on $\overline{X_s}$ and such that $i_s^* \cT \cong \underline{\O}{}_{X_s} [\dim X_s]$.

\end{prop}

\begin{proof}
Assume (without loss of generality) that $X=\overline{X_s}$. For any subset $I \subset \scS$, we set $X_I:= \sqcup_{t \in I} X_t$. We will prove the following property by induction on $\#I$:
\begin{equation} \label{hyp:tilting}
\begin{array}{c}
\text{for $I \subset \scS$ containing $s$ with $X_I \subset X$ open there exists a tilting} \\ 
\text{perverse sheaf $\cT_I$ in $\Perv_{\scS}(X_I,\O)$ extending $\underline{\O}{}_{X_s} [\dim X_s]$.}\end{array}
\end{equation}
The proposition is the case $I=\scS$ of this property. Taking $\cT = \underline{\O}{}_{X_s} [\dim X_s]$ shows that \eqref{hyp:tilting} is
satisfied with $I = \{ s \}$.

So assume that $\{s\} \subset I \subset \scS$ is arbitrary (with $X_I \subset X$ open), choose $t \in I$ such that $X_t \subset X_I$ is a closed stratum and write $I = J \cup \{ t
\}$. By induction there exists a perverse sheaf $\cT_{J}$ which satisfies \eqref{hyp:tilting} for $J$.

Let $j \colon X_{J} \hookrightarrow X_I$ denote the (open) inclusion and
set $\cT_I^{\mathrm{pre}} := j_! \cT_{J}$. Note that $\cT_I^{\mathrm{pre}}$ is a perverse sheaf since it has a filtration (in the triangulated sense) with subquotients $\Delta_u$, $u \in J$. Now set
\[
E := \Ext^1_{\Perv_{\scS}(X_I,\O)}(\Delta_{I,t}, \cT_I^{\mathrm{pre}}) = \Ext^1_{\Db_{\scS}(X_I,\O)}(\Delta_{I,t}, \cT_I^{\mathrm{pre}})
\]
which is a finitely generated $\O$-module. (Here, in a minor abuse of notation, we still denote by $\scS$ the stratification of $X_I$ induced by $\scS$, and by $\Delta_{I,t}$ the shifted $!$-push-forward of $\underline{\O}{}_{X_t}$ to $X_I$, i.e.~the standard object of $\Perv_{\scS}(X_I,\O)$ associated with~$t$.) Let
$E_{\mathrm{free}}$ be a finitely generated free $\O$-module endowed with a surjection to $E$, and set $E_{\mathrm{free}}^*:=\Hom_{\O}(E_{\mathrm{free}},\O)$. We have a natural map
\begin{multline*}
\O \to E^*_{\mathrm{free}} \otimes_{\O} E = E^*_{\mathrm{free}}
\otimes_{\O} \Ext^1_{\Perv_{\scS}(X_I,\O)}(\Delta_{I,t}, \cT_I^{\mathrm{pre}}) \\
\cong \Ext^1_{\Perv_{\scS}(X_I,\O)}(E_{\mathrm{free}} \otimes_{\O} \Delta_{I,t}, \cT_I^{\mathrm{pre}}) 
\end{multline*}
and hence a canonical extension in $\Perv_{\scS}(X_I,\O)$
\begin{equation} \label{eq:tiltingextension}
\cT_I^{\mathrm{pre}} \hookrightarrow \cT_I \twoheadrightarrow
E_{\mathrm{free}} \otimes_{\O} \Delta_{I,t}
\end{equation}
obtained as the image of $1 \in \O$.

We claim that $\cT_I$ is tilting. First, this object clearly has a standard filtration (as an object of $\Perv_{\scS}(X_I,\O)$). Now we show (using Lemma \ref{lem:nabla-flag}) that it has a costandard filtration. For any $u \in I$ we denote by $i_u^I \colon X_u \hookrightarrow X_I$ the inclusion, and similarly for $J$. For $u \neq t$, it is easy to check that $(i_u^I)^! \cT_I$ is a direct sum of copies of $\underline{\O}{}_{X_u}[\dim X_u]$: indeed in this case we have $i_u^I=j \circ i_u^J$, 
so $(i_u^I)^! \cT_I = (i_u^J)^! j^! \cT_I = (i_u^J)^! \cT_{J}$. We conclude using the fact that $\cT_J$ has a costandard filtration.

Now, consider the case $u=t$. We claim that
\begin{equation} \label{eq:cond}
\Ext^i_{\Db_\scS(X_I,\O)}(\Delta_{I,t}, \cT_I) = 0 \ \text{ for } i\geq 1.
\end{equation}
We first show that \eqref{eq:cond} holds for $i=1$. Applying $\Hom(\Delta_{I,t}, -)$ to \eqref{eq:tiltingextension} yields a
long exact sequence
\begin{multline*}
\dots \to \Hom_{\Db_\scS(X_I,\O)}(\Delta_{I,t}, E_{\mathrm{free}} \otimes_{\O} \Delta_{I,t}) \to 
\Ext^1_{\Db_\scS(X_I,\O)}(\Delta_{I,t}, \cT_I^{\mathrm{pre}}) \\
\to 
\Ext^1_{\Db_\scS(X_I,\O)}(\Delta_{I,t}, \cT_I) \to \Ext^1_{\Db_\scS(X_I,\O)}(\Delta_{I,t}, E_{\mathrm{free}} \otimes_{\O} \Delta_{I,t}) \to \dots 
\end{multline*}
Now $\Ext^1_{\Db_\scS(X_I,\O)}(\Delta_{I,t}, \Delta_{I,t}) = 0$, so it only remains to see that the
first arrow is surjective. But under the canonical identification 
\[
\Hom_{\Db_\scS(X_I,\O)}(\Delta_{I,t}, E_{\mathrm{free}} \otimes_{\O} \Delta_{I,t}) = E_{\mathrm{free}}
\otimes_{\O} \Hom_{\Db_\scS(X_I,\O)}(\Delta_{I,t}, \Delta_{I,t}) = E_{\mathrm{free}}
\]
this map corresponds to the surjection $E_{\mathrm{free}} \twoheadrightarrow E = \Ext^1_{\Db_\scS(X_I,\O)}(\Delta_{I,t},  \cT_I^{\mathrm{pre}})$ chosen above. Hence this
map is indeed surjective.

We now turn to \eqref{eq:cond} for $i \geq 2$. Because $\cT_{J}$ has a costandard filtration, $\cT_I^{\mathrm{pre}}=j_! \cT_{J}$ has a filtration (in the triangulated sense) by objects of the form
$j_!\nabla_{J,u}$ for $u \in J$. (Note that $j_!\nabla_{J,u}$ need not be a perverse sheaf.) Hence it is enough to show that
\[
\Hom^i_{\Db_\scS(X_I,\O)}(\Delta_{I,t}, j_! \nabla_{J,u}) = 0 \quad \text{for all $u \in J$ and $i \geq 2$.}
\]
Consider the distinguished triangles in $\Db_{\scS}(X_I,\O)$
\begin{equation}
\label{eq:extses}
\cM \to j_! \nabla_{J,u} \to
\IC(X_{J}, \nabla_{J,u}) \xrightarrow{[1]}, \qquad
\IC(X_{J}, \nabla_{J,u}) \to j_* \nabla_{J,u} \to
\cN \xrightarrow{[1]}.
\end{equation}
(Here, $\IC(X_J,-)$ is the intermediate extension for the embedding $j$.)
Note that both $\cM$ and $\cN$ are supported on
$X_t$. Moreover, $\cN$ is perverse, and $\cM$ is concentrated in nonpositive perverse degrees. In particular, both of these objects are direct sums of objects of the form $(i^I_t)_! \mathcal{L}[\dim X_t +i]$, where $\mathcal{L}$ is a constant local system on $X_t$ and $i \geq 0$. It follows that 
$
\Hom^k_{\Db_\scS(X_I,\O)}(\Delta_{I,t}, \cM) = \Hom^k_{\Db_\scS(X_I,\O)}(\Delta_{I,t}, \cN) = 0
$
for $k > 0$. Applying $\Hom(\Delta_{I,t}, -)$ to the second triangle in
\eqref{eq:extses} and using the fact that 
\[
\Hom^i_{\Db_\scS(X_I,\O)}(\Delta_{I,t}, j_*\nabla_{J,u}) = \Hom^i_{\Db_\scS(X_I,\O)}(\Delta_{I,t}, \nabla_{I,u}) = 0
\]
for $i \ge 1$, we get that $\Ext^i_{\Db_\scS (X_I,\O)} \bigl( \Delta_{I,t}, \IC(X_{J}, \nabla_{J,u}) \bigr) = 0$ for $i \ge 2$. Then applying
the same functor to the first triangle in~\eqref{eq:extses} we get $\Ext^i_{\Db_\scS (X_I,\O)}(\Delta_{I,t},
 j_! \nabla_{J,u})= 0$ as claimed.

Property \eqref{eq:cond} implies that $(i_t^I)^! \cT_I$ is a perverse sheaf. To show that it is a shifted \emph{free} local system, it is enough to prove that $\F \bigl( (i_t^I)^! \cT_I \bigr)=(i_t^I)^! \F(\cT_I)$ is in nonnegative perverse degrees. However, $\F(\cT_I)$ is a perverse sheaf since $\cT_I$ has a standard filtration, and $(i_t^I)^!$ is left exact. Hence indeed $\pH^k \F \bigl( (i_t^I)^! \cT_I \bigr)=0$ for $k<0$, which finishes the proof.
\end{proof}

\subsection{Tilting perverse sheaves: properties}

In the case of coefficients $\F$, it is well known that
if $s \in \scS$ there exists a unique indecomposable tilting perverse
sheaf $\cT_{s}^\F$ (up to isomorphism) which is supported on
$\overline{X_s}$, and such that $i_s^* \cT_{s}^\F = \underline{\F}_{X_s}
[\dim X_s]$. Moreover, any tilting perverse sheaf in
$\Perv_{\scS}(X,\F)$ is a direct sum of objects $\cT_{s}^\F$ for $s \in
\scS$. (Although the case of coefficients $\F$ is not considered in
\cite{bbm}, the proofs generalize to this case. Alternatively, one can
can use the theory of highest weight categories, see~\cite{ringel}.)

\begin{prop}
\label{prop:properties-tilting}

\begin{enumerate}
\item For any tilting perverse sheaf $\cT$ in $\Perv_{\scS}(X,\O)$, $\F(\cT)$ is a perverse sheaf. It is tilting in $\Perv_{\scS}(X,\F)$.
\item If $\cT$ and $\cT'$ are tilting perverse sheaves in $\Perv_{\scS}(X,\O)$, the $\O$-module
\[
\Hom_{\Perv_{\scS}(X,\O)}(\cT,\cT')
\]
is free. Moreover, the natural morphism
\[
\F \otimes_{\O} \Hom_{\Perv_{\scS}(X,\O)}(\cT,\cT') \ \to \ \Hom_{\Perv_{\scS}(X,\F)} \bigl( \F(\cT), \F(\cT') \bigr)
\]
is an isomorphism.
\item A tilting perverse sheaf $\cT$ in $\Perv_{\scS}(X,\O)$ is indecomposable iff $\F(\cT)$ is indecomposable.
\item For any $s \in \scS$, there exists a unique indecomposable tilting perverse sheaf $\cT_s$ in $\Perv_{\scS}(X,\O)$ supported on $\overline{X_s}$ and such that $i_s^* \cT_s \cong \underline{\O}_{X_s} [\dim X_s]$. We have $\F(\cT_s) \cong \cT_{s}^\F$, and any tilting perverse sheaf in $\Perv_{\scS}(X,\O)$ is a direct sum of objects $\cT_s$ ($s \in \scS$).

\end{enumerate}

\end{prop}

\begin{proof}
The proof of $(1)$ is immediate from~\eqref{eqn:F-Delta-nabla}.
$(2)$ follows from the facts that $\cT$ admits a standard filtration, and that $\cT'$ admits a costandard filtration, together with the property that
\[
\Ext^i_{\Db_\scS(X,\O)}(\Delta_s,\nabla_t)=0 \ \text{ if } i>0,
\]
and similarly for coefficients $\F$.
Then $(3)$ and $(4)$ can be proved as in \cite[Corollary 2.4.2]{rsw}.
\end{proof}

\begin{rmk}
From the uniqueness statement in Proposition \ref{prop:properties-tilting}$(4)$ one deduces that $\mathbb{D}_X(\cT_s) \cong \cT_s$ for any $s \in \scS$, where $\mathbb{D}_X$ denotes Verdier duality.
\end{rmk}

\subsection{Extension of scalars}

The results of this subsection are not used in this paper, but are needed in \cite{ar, ar2}. For simplicity, here we restrict to the case where $X$ is defined over $\C$.

Tilting perverse sheaves can also be defined when the coefficients are $\K$. We use similar notation as above in this case. We denote by $\Tilt_\scS(X,\E) \subset \Perv_\scS(X,\E)$ the additive full subcategory whose objects are tilting perverse sheaves.

\begin{lem}
The natural functors
\[
\Kb \Tilt_{\scS}(X,\E) \to \Db \Perv_\scS(X,\E) \to \Db_\scS(X,\E).
\]
are equivalences of categories.
\end{lem}
\begin{proof}[Proof sketch]
For $\E = \K$, this was proved in~\cite[Proposition~1.5]{bbm}.  The same proof applies verbatim for $\E = \F$.  Thanks to the results above (see also~\cite[Corollary~2.3.4]{rsw}), these arguments go through for $\E = \O$ as well.
\end{proof}

The preceding lemma makes it possible to interpret extension of scalars as the derived functor of a functor between categories of perverse sheaves.  Because $\F({-}) \colon \Db_\scS(X,\O) \to \Db_\scS(X,\F)$ is right t-exact, it gives rise to a right exact functor of abelian categories $\F^0 := \pH^0 \circ \F({-}) \colon \Perv_\scS(X,\O) \to \Perv_\scS(X,\F)$.  Since $\Perv_\scS(X,\O)$ has enough projectives (see~\cite[Corollary~2.3.3]{rsw}), we can form its left derived functor $L \F^0$.  On the other hand, by Proposition~\ref{prop:properties-tilting}(1), $\F^0$ restricts to an additive functor $\Tilt_\scS(X,\O) \to \Tilt_\scS(X,\F)$.

\begin{lem}
The following diagram commutes up to isomorphism of functors:
\[
\xymatrix@C=1.5cm{
\Kb \Tilt_\scS(X,\O) \ar[r]^-{\sim} \ar[d]_{\Kb \F^0} &
  \Db \Perv_\scS(X,\O) \ar[r]^-{\sim} \ar[d]_{L \F^0} &
  \Db_\scS(X,\O) \ar[d]^{\F({-})} \\
\Kb \Tilt_\scS(X,\F) \ar[r]^-{\sim} &
  \Db \Perv_\scS(X,\F) \ar[r]^-{\sim} &
  \Db_\scS(X,\F) }
\]
There is a similar commutative diagram for $\K({-})$.
\end{lem}
\begin{proof}
For the right-hand square of this diagram, by adjunction, we may instead check the commutativity of the corresponding diagram for the restriction-of-scalars functor $\Db_\scS(X,\F) \to \Db_\scS(X,\O)$ coming from the map $\O \to \F$. The latter functor is exact on (ordinary, not perverse) sheaves, so it lifts to a suitable ``filtered'' version of $\Db_\scS(X,\F)$. In other words, it conforms to the setting of~\cite[Lemma~A.7.1]{beilinson}, which asserts the desired commutativity.

We can now recast Proposition~\ref{prop:properties-tilting}(1) as follows: if $\cT \in \Perv_\scS(X,\O)$ is tilting, then $L\F^0(\cT)$ is perverse and tilting.

For the left-hand square above, commutativity is a consequence of general properties of derived functors.  In more detail, recall that the derived functor $L\F^0$ comes equipped with a natural transformation $\varepsilon \colon L\F^0 \circ Q_\O \to Q_\F \circ \Kb \F^0$, where $Q_\E \colon \Kb \Perv_\scS(X,\E) \to \Db\Perv_\scS(X,\E)$ is the obvious functor for $\E=\O$ or $\F$.  Moreover, for any object $\cF \in \Perv_\scS(X,\O)$ such that $L\F^0(\cF)$ is perverse, the morphism $\varepsilon_{\cF} \colon L\F^0(Q_\O(\cF)) \to Q_\F(\Kb \F^0(\cF))$ is an isomorphism. In particular, if $\cT$ is tilting, then $\varepsilon_{\cT}$ is an isomorphism. Since $\Kb\Tilt_\scS(X,\O)$ is generated as a triangulated subcategory of $\Kb\Perv_\scS(X,\O)$ by tilting perverse sheaves, $\varepsilon$ becomes an isomorphism when we restrict the domain of $Q_\E$ to $\Kb\Tilt_\scS(X,\E)$.

The argument for $\K({-})$ is similar and will be omitted.
\end{proof}

\subsection{Radon transform}
\label{ss:radon}

In this subsection we recall the formalism of Radon transform (see \cite{bb, bbm, yun}) and check that it generalizes to coefficients in $\O$. We follow the approach in \cite{yun} closely.

We let $B$ be an algebraic group with maximal torus $T$ and
consider two $B$-varieties $X$ and $Y$ with finitely many $B$-orbits,
each of which is isomorphic to an affine space.
We denote by
\[
X = \bigsqcup_{s \in \scS} X_s, \qquad Y = \bigsqcup_{t \in \scT} Y_t
\]
the stratifications by $B$-orbits. 
In the {\'e}tale case, 
these stratifications automatically satisfy condition~\eqref{eqn:strata-assumption}. 
We fix an open $B$-stable subvariety $U \subset X
\times Y$, and denote the natural morphisms as follows:
\[
\xymatrix{
X & U \ar[l]_-{\overleftarrow{u}} \ar[r]^-{\overrightarrow{u}} & Y.\\
}
\]
We will assume that the following
conditions are satisfied (see \cite[\S 4.1]{yun}):
\begin{enumerate}
\item for any $s \in \scS$ (resp.~$t \in \scT$), $X_s$ (resp.~$Y_t$) contains a unique $T$-fixed point $x_s$ (resp.~$y_t$);
\item for each $s \in \scS$ (resp.~$t \in \scT$), the open subset $Y^s:=\overrightarrow{u}(\overleftarrow{u}^{-1}(x_s)) \subset Y$ (resp.~$X^t:=\overleftarrow{u}(\overrightarrow{u}^{-1}(y_t)) \subset X$) contains a unique $T$-fixed point $y_{\hat{s}}$ for some $\hat{s} \in \scT$ (resp.~$x_{\hat{t}}$ for some $\hat{t} \in \scS$) and contracts to that fixed point under some one-parameter subgroup $\Gm \subset T$ (depending on $s$ or $t$);
\item for each $s \in \scS$ we have $\dim X_s=\codim_Y Y_{\hat{s}}$.
\end{enumerate}
It follows in particular from these assumptions that $s \mapsto \hat{s}$ and $t \mapsto \hat{t}$ are inverse bijections between $\scS$ and $\scT$. Note also that if $\dim X=\dim Y$ then these assumptions are symmetric in $X$ and $Y$. For $s \in \scS$ and $t\in \scT$, we denote by $i_s^X : X_s \hookrightarrow X$ and $i_t^Y : Y_t \hookrightarrow Y$ the inclusions.

\begin{rmk}
Our assumptions are satisfied in the setting of~\S\ref{ss:reminder-radon} by \cite[\S5.1]{yun}.
\end{rmk}

We set
\begin{multline*}
\mathsf{R}_{X \to Y} := \overrightarrow{u}_! \overleftarrow{u}^* [\dim Y] \colon \Db_{\scS}(X,\O) \to \Db_{\scT}(Y,\O), \\
\mathsf{R}_{X \leftarrow Y} := \overleftarrow{u}_* \overrightarrow{u}^! [-\dim Y] \colon \Db_{\scT}(Y,\O) \to \Db_{\scS}(X,\O).
\end{multline*}

\begin{prop}
\label{prop:radon-nabla-delta}

For any $s \in \scS$, $t \in \scT$ there exist isomorphisms
\begin{equation*}
\mathsf{R}_{X \to Y}(\nabla_{s}) \cong \Delta_{\hat{s}}, \qquad
\mathsf{R}_{X \leftarrow Y}(\Delta_{t}) \cong \nabla_{\hat{t}}.
\end{equation*}

\end{prop}

\begin{proof}
This proof is copied from \cite[Proposition 4.1.3]{yun}. We only prove the first isomorphism; the proof of the second one is similar. For any $v \in \scT$ (resp.~$w \in \scS$), let us denote by $j_v^Y \colon \{y_v\} \hookrightarrow Y$ (resp.~$j_w^X \colon \{x_w\} \hookrightarrow X$) the inclusion. As $\mathsf{R}_{X \to Y}(\nabla_{s})$ is $\scT$-constructible, it is sufficient to show that for $v \in \scT$ we have
\begin{equation}
\label{eqn:fiber-radon}
(j_v^Y)^* \mathsf{R}_{X \to Y}(\nabla_{s}) \cong \begin{cases}
\O [\dim Y_u] & \text{if $v=\hat{s}$;} \\
0 & \text{if $u \neq \hat{s}$.}
\end{cases}
\end{equation}
By definition and the proper base change theorem we have
\begin{multline*}
(j_v^Y)^* \mathsf{R}_{X \to Y}(\nabla_{s}) = (j_v^Y)^* \overrightarrow{u}_! \overleftarrow{u}^* (\nabla_{s}) [\dim Y] \\
\cong \mathbb{H}^\bullet_c \bigl( \overrightarrow{u}^{-1}(y_v), \overleftarrow{u}^* (\nabla_{s})_{| \overrightarrow{u}^{-1}(y_v)} \bigr) [\dim Y].
\end{multline*}
Hence, again by definition, we obtain
\[
(j_v^Y)^* \mathsf{R}_{X \to Y}(\nabla_{s}) \cong \mathbb{H}^\bullet_c(X^v, \nabla_s{}_{|X^v}) [\dim Y].
\]
Now, using assumption $(2)$ and \cite[Proposition 1]{soergel-cohom} we have
\[
\mathbb{H}^{\bullet}_c(X^v, \nabla_s{}_{|X^v}) [\dim Y] \cong k_{\hat{v}}^! \nabla_s{}_{|X^v} [\dim Y],
\]
where $k_{\hat{v}} \colon \{x_{\hat{v}}\} \hookrightarrow X^v$ is the inclusion. As $X^v \subset X$ is open, we finally obtain an isomorphism
\[
(j_v^Y)^* \mathsf{R}_{X \to Y}(\nabla_{s}) \cong (j_{\hat{v}}^X)^! \nabla_{s} [\dim Y].
\]
Now, let $l_{\hat{v}} \colon \{x_{\hat{v}}\} \hookrightarrow X_{\hat{v}}$ be the inclusion. We have
\[
(j_{\hat{v}}^X)^! \nabla_{s} [\dim Y] \cong l_{\hat{v}}^! (i_{\hat{v}}^X)^! \nabla_{s} [\dim Y].
\]
If $s \neq \hat{v}$ (or equivalently if $v \neq \hat{s}$) the right-hand side is zero, which proves \eqref{eqn:fiber-radon} in this case. If $s=\hat{v}$, the right-hand side identifies with
\[
l_{\hat{v}}^! \underline{\O}{}_{X_{\hat{v}}} [\dim Y +\dim X_{\hat{v}}] \cong \O [\dim Y -\dim X_{\hat{v}}].
\]
Using assumption $(3)$, this finishes the proof of \eqref{eqn:fiber-radon}.
\end{proof}

\begin{cor}
\label{cor:radon}

\begin{enumerate}
\item The functors $\mathsf{R}_{X \to Y}$ and $\mathsf{R}_{X \leftarrow Y}$ are quasi-inverse equivalences of categories between $\Db_{\scS}(X,\O)$ and $\Db_{\scT}(Y,\O)$.
\item For any tilting perverse sheaf $\cT$ in $\Perv_{\scS}(X,\O)$, the object $\mathsf{R}_{X \to Y}(\cT)$ is a projective perverse sheaf in $\Perv_{\scT}(Y,\O)$.
\item For any projective perverse sheaf $\cP$ in $\Perv_{\scT}(Y,\O)$, the object $\mathsf{R}_{X \leftarrow Y}(\cP)$ is a tilting perverse sheaf in $\Perv_{\scS}(X,\O)$.
\end{enumerate}

\end{cor}

\begin{proof}
(1) This proof is copied from \cite[Corollary 4.1.5]{yun}. Consider the morphism 
\begin{equation}
\label{eqn:adjunction}
\id \to \mathsf{R}_{X \leftarrow Y} \circ \mathsf{R}_{X \to Y}
\end{equation}
defined by adjunction.
By Proposition~\ref{prop:radon-nabla-delta}, this morphism is an isomorphism when applied to any object $\nabla_{s}$ ($s \in \scS$). By an easy induction on the number of strata, one can check that these objects generate the category $\Db_{\scS}(X,\O)$, which proves that \eqref{eqn:adjunction} is an isomorphism. A similar argument proves that the adjunction morphism $\mathsf{R}_{X \to Y} \circ \mathsf{R}_{X \leftarrow Y} \to \id$ is also an isomorphism.

(2) This proof is copied from \cite[Proposition 4.2.1]{yun}. Set $\cP=\mathsf{R}_{X \to Y}(\cT)$. As $\cT$ has a costandard filtration, by Proposition~\ref{prop:radon-nabla-delta} $\cP$ is in the subcategory of $\Db_{\scT}(Y,\O)$ generated by objects $\Delta_t$ under extensions, hence is a perverse sheaf. Now, let us show that it is projective, i.e.~that $\Hom_{\Db_\scT(Y,\O)}(\cP,\cM)=0$ for any $\cM \in {}^p\!\Db_{\scT}(Y,\O)^{<0}$. As ${}^p\!\Db_{\scT}(Y,\O)^{< 0}$ is generated (under extensions) by objects $\Delta_{t}[m]$ for $t \in \scT$, $m \in \Z_{> 0}$, it is sufficient to prove that $\Ext_{\Db_\scT(Y,\O)}^i(\cP,\Delta_{t})=0$ for $i >0$. Now we have, using Proposition~\ref{prop:radon-nabla-delta},
\[
\Ext_{\Db_\scT(Y,\O)}^i(\cP,\Delta_{t}) \cong \Ext_{\Db_\scS(X,\O)}^i(\cT,\mathsf{R}_{X \leftarrow Y}(\Delta_{t})) \cong \Ext_{\Db_\scS(X,\O)}^i(\cT,\nabla_{\hat{t}}),
\]
which proves the claim since $\cT$ has a standard filtration.

(3) The proof is similar to that of (2). Set $\cT=\mathsf{R}_{X \leftarrow Y}(\cP)$. As $\cP$ has a standard filtration, by Proposition~\ref{prop:radon-nabla-delta} $\cT$ is perverse and has a costandard filtration. It follows that $\D_X(\cT)$ is also perverse. (Here $\mathbb{D}_X$ is Verdier duality.) To prove that $\cT$ has a standard filtration, it is enough to prove that $\D_X(\cT)$ has a costandard filtration. By Lemma \ref{lem:nabla-flag}, this would follow if we can prove that $\Ext^i_{\Db_\scS(X,\O)}(\Delta_s,\D(\cT))=0$ for $i > 0$ and is $\O$-free for $i=0$. This follows from the chain of isomorphisms
\begin{multline*}
\Ext^i_{\Db_\scS(X,\O)}(\Delta_s,\D_X(\cT)) \cong \Ext^i_{\Db_\scS(X,\O)}(\cT,\nabla_s) \cong \Ext^i_{\Db_\scT(Y,\O)}(\cP,\mathsf{R}_{X \to Y}(\nabla_s)) \\
\cong \Ext^i_{\Db_\scT(Y,\O)}(\cP,\Delta_{\hat{s}})
\end{multline*}
(see (1) and Proposition~\ref{prop:radon-nabla-delta}), using Lemma~\ref{lem:Hom-proj-F} and the fact that $\cP$ is projective.
\end{proof}

\end{document}